\documentclass[11pt]{amsart}
\usepackage[utf8]{inputenc}
\usepackage{amsmath}
\usepackage{scalerel}
\usepackage{standalone}

\usepackage{hyperref}
\hypersetup{
    colorlinks=true,
    linkcolor=purple,
    citecolor=teal,      
    urlcolor=purple,
}
\usepackage{amsmath,amssymb,amsthm,mymathabx}
\usepackage{dsfont}
\usepackage[dvipsnames]{xcolor} 
\usepackage{tikz}
    \usetikzlibrary{cd}
\usepackage{tikz-cd}
\allowdisplaybreaks
\usepackage{comment}

\usepackage[all]{xy}

\usepackage{wasysym}

\theoremstyle{plain}
\newtheorem{theorem}{{Theorem}}[section]
\newtheorem{lemma}[theorem]{{Lemma}}
\newtheorem{proposition}[theorem]{{Proposition}}

\newcommand{\R}{\mathbb{R}}
\newcommand{\Z}{\mathbb{Z}}
\newcommand{\BF}{\mathbb{F}} 
\newcommand{\CA}{\mathcal{A}} 
\newcommand{\essl}{\mathfrak{sl}}

\newcommand{\map}[1]{\xrightarrow{#1}}

\newcommand{\into}{\hookrightarrow}

\newcommand{\wo}{\backslash}
\newcommand{\resp}{resp.\ }
\newcommand{\doublebar}[1]{\overline{\overline{#1}}}

\newcommand{\abuts}{\rightrightarrows} 

\newcommand{\ovl}{\overline}
\newcommand{\Id}{\mathbb{I}} 
 
\newcommand{\bdy}{\partial}
\newcommand{\cone}{\mathrm{Cone}}

\newcommand{\Kh}{\mathit{Kh}}
\newcommand{\CKh}{\mathit{CKh}} 
\newcommand{\Khred}{\widetilde{\Kh}}
\newcommand{\CKhred}{\widetilde{\CKh}} 

\newcommand{\BN}{{\mathit{BN}^2}} 
\newcommand{\CBN}{{\mathit{CBN}^2}} 
\newcommand{\BNred}{\widetilde{\mathit{BN}}^2} 
\newcommand{\CBNred}{\widetilde{\mathit{CBN}}^2} 

\newcommand{\HFhat}{\widehat{\mathit{HF}}}
\newcommand{\HFplus}{\mathit{HF}^+}
\newcommand{\CFI}{\mathit{CFI}}

\newcommand{\CFhat}{\widehat{\mathit{CF}}}
\newcommand{\CFIhat}{\widehat{\mathit{CFI}}}
\newcommand{\HFIhat}{\widehat{\mathit{HFI}}}
\newcommand{\HFI}{\mathit{HFI}}
\newcommand{\HD}{\mathcal{H}}
\newcommand{\fraks}{\mathfrak{s}} 
\newcommand{\alphas}{\boldsymbol{\alpha}}
\newcommand{\betas}{\boldsymbol{\beta}}

\newcommand{\Hext}{\mathcal{H}} 

\newcommand{\vv}{\mathbf{v}}

\newcommand{\CFAhat}{\widehat{\mathit{CFA}}}
\newcommand{\CFDhat}{\widehat{\mathit{CFD}}}
\newcommand{\CFDAhat}{\widehat{\mathit{CFDA}}}
\newcommand{\CFDDhat}{\widehat{\mathit{CFDD}}}
\newcommand{\CFAAhat}{\widehat{\mathit{CFAA}}}
\newcommand{\CFDAAhat}{\widehat{\mathit{CFDAA}}}

\newcommand{\DT}{\boxtimes} 
\newcommand{\bt}{\mid} 
\newcommand{\idmap}{\mathrm{Id}} 

\newcommand{\AZ}{\mathsf{AZ}}
\newcommand{\AZbar}{\overline{\AZ}}
\newcommand{\ZZ}{\mathcal{Z}}

\newcommand{\h}{{\Psi}} 
\newcommand{\hcomp}{{\psi}} 

\newcommand{\ro}[1]{\rho_{#1}} 
\newcommand{\io}[1]{\iota_{#1}} 
\newcommand{\rod}[1]{\rho^*_{#1}} 
\newcommand{\iod}[1]{\iota^*_{#1}} 

\newcommand{\flip}[1]{
    \reflectbox{\rotatebox[origin=c]{180}{#1}}
    }
\usetikzlibrary{shapes.multipart}
 
\newcommand{\deltaAZAZ}{\delta^{\scaleto{\AZbar\cup\AZ}{5pt}}}
\usetikzlibrary {arrows,bending,positioning}

\newcommand{\F}{\mathcal{F}} 

\usepackage{todonotes}

\usepackage{tikz-cd} 

\usepackage[margin=1in]{geometry}

\allowdisplaybreaks

\title[Khovanov homology and Involutive Heegaard Floer homology]{Khovanov homology and the Involutive Heegaard Floer homology of branched double covers}

\author[A. Alishahi]{Akram Alishahi}
\address{Department of Mathematics, University of Georgia, Athens, GA 30602, U.S.A.}
\email{akram.alishahi@uga.edu}

\author[L. Truong]{Linh Truong}
\address{Department of Mathematics, University of Michigan, Ann Arbor, MI 48109, U.S.A.}
\email{tlinh@umich.edu}

\author[M. Zhang]{Melissa Zhang}
\address{Department of Mathematics, UC Davis, One Shields Ave., Davis, CA 95616-8633, U.S.A.}
\email{\href{mailto:mlzhang@ucdavis.edu}{mlzhang@ucdavis.edu}}

\date{\today}

\begin{document}

\maketitle

\begin{abstract}
    We use involutive Heegaard Floer homology to extend the Ozsv\'ath-Szab\'o branched double cover spectral sequence relating a version of Khovanov homology and the Heegaard Floer homology of branched double covers. Our main tools are Lipshitz, Ozsv\'ath, and Thurston's reconstruction of the Ozsv\'ath-Szab\'o spectral sequence using bordered Floer homology and Hendricks and Lipshitz's surgery exact triangle in involutive bordered Floer homology. 
\end{abstract}

\tableofcontents

\section{Introduction}

The interactions between the link invariants from Khovanov homology and Heegaard Floer homology have been studied extensively over the past two decades. The first relation is a spectral sequence from the reduced Khovanov homology to the Heegaard Floer homology of the branched double cover of the link constructed by Ozsv\'{a}th and Szab\'{o} \cite{OzsvathSzaboss} in 2005. Gauge theoretic versions of this spectral sequence have been constructed for monopole Floer homology by Bloom in \cite{Bloom} and for instanton Floer homology by Kronheimer and Mrowka in \cite{KM:unknot}.  In fact, Kronheimer and Mrowka use their spectral sequence to prove that Khovanov homology detects the unknot. 

In \cite{Manolescu-triangulation}, Manolescu used the $\mathrm{Pin}(2)$-symmetry on Seiberg-Witten Floer homology to define a $\mathrm{Pin}(2)$-equivariant theory, and used this to disprove the Triangulation Conjecture.  En route to constructing a Heegaard Floer analogue of $\mathrm{Pin}(2)$-equivariant Seiberg-Witten Floer homology, Hendricks and Manolescu defined \emph{involutive Heegaard Floer homology}. This theory found topological applications in many areas, including the study of the homology cobordism group \cite{DHST, HHSZ}, and the answer to Kawauchi's forty-year-old question about the smooth sliceness of the $(2,1)$-cable of the figure-eight knot \cite{DKMPS}.

In 2016, Lin introduced involutive monopole Floer homology \cite{Lin}, as an analogue of involutive Heegaard Floer homology of Hendricks and Manolescu \cite{HendricksManolescu}.
Moreover, he constructed a spectral sequence converging to involutive monopole Floer homology of the branched double cover of a link.  The $E^2$ page of this spectral sequence is a  Bar-Natan homology \cite{BarNatan}, a variant of Khovanov homology similarly defined and combinatorially (i.e.\ algorithmically) constructed. 

Our main theorem is the analogue of Lin's result in the setting of involutive Heegaard Floer homology. For a link $L \subset S^3$, we let $\overline{L}$ denote the mirror of the link $L$, and let $\Sigma(L)$ denote the branched double cover of $S^3$ with branch locus $L$. 
We construct a spectral sequence relating a version of reduced Bar-Natan homology $\BNred(\overline{L})$ of $\overline{L}$ with the involutive Heegaard Floer homology of the branched double cover $\Sigma(L)$. 

\begin{theorem}
For any link $L \subset S^3$, there is a spectral sequence of $\BF[Q]/(Q^2)$-modules whose $E^2$ page is isomorphic to $\BNred(\overline{L})$ and which converges to involutive Heegaard Floer homology $\HFIhat(\Sigma(L))$. 
\end{theorem}

Our main tools are Lipshitz, Ozsv\'ath, and Thurston's reconstruction of the Ozsv\'ath-Szab\'o spectral sequence using bordered Floer homology \cite{LOTborderedss, LOTborderedss2}
and Hendricks and Lipshitz's surgery exact triangle in involutive bordered Floer homology \cite{HendricksLipshitz}. 

\subsection{Organization}

In \S 2, we briefly review the relevant versions of Khovanov and Heegaard Floer homology, in order to set the notation going forward.
In \S 3, we compute explicit homotopy representatives of the maps and homotopies that appear when we construct the spectral sequence in \S 4. In \S 5, we prove that, like the Ozsv\'ath-Szab\'o spectral sequence, our involutive link surgeries spectral sequence also collapses immediately on Khovanov-thin knots; we also show that our spectral sequence is nontrivial for the the $(3,7)$-torus knot, whose mirror's branched double cover is $(+1)$-surgery on the left-handed trefoil.

The Python code for all our computer-assisted verifications is available at \cite{math-code}. 

\subsection{Acknowledgements} We thank Kristen Hendricks, Robert Lipshitz, and Liam Watson for helpful conversations. AA was partially supported by NSF grant DMS-2000506. LT was partially supported by NSF grant DMS-2104309.
A portion of this article was completed while AA and MZ were supported by the National Science Foundation under Grant No.~DMS-1928930, while they were in residence at the Simons Laufer Mathematical Sciences Institute (previously known as MSRI) in Berkeley, California, during the Fall 2022 Floer Homotopy Theory semester.

\section{Background}

We assume some familiarity with the hat version of the Heegaard Floer homology  $\HFhat(Y^3)$ of a 3-manifold $Y^3$, as well as ordinary Khovanov homology $\Kh(L)$ of a link $L$ in $S^3$. We work with $\BF = \BF_2$ coefficients exclusively. For more detailed introductions on these constructions, see for instance \cite{OzsvathSzabo-intro} for $\HFhat$, and \cite{BarNatanKh} for $\Kh$.

In this section, we briefly introduce the relevant variations on $\HFhat$ and $\Kh$ for our purposes, as well as some important tools that we will use. This serves partly to set the notation for the remainder of the article.

\subsection{Khovanov and Bar-Natan Homology}

Introduced in \cite{Khovanov}, Khovanov homology assigns to a link $L \subset S^3$ a bigraded $\BF$-vector space 
\[\Kh(L) = \bigoplus_{h,q \in \Z} \Kh^{h,q}(L).\]
Here, the symbol $\Kh^{h,q}(L)$ represents the vector space at homological grading $h$ and quantum grading $q$.

Bar-Natan homology, introduced in \cite{BarNatan}, is a generalization of Khovanov homology 
defined by perturbing the Frobenius algebra. We will focus on a specific version of Bar-Natan homology with coefficients in the ring $\BF[Q]/(Q^2)$ which also appears in \cite{Lin}.

\subsubsection{Cube of resolutions}
Let $L$ be a link in $S^3$, and let $D$ be a diagram for $L$ with $n$ crossings. Fix an ordering on the crossings.

Each crossing in $L$ can be resolved in two ways: 
\begin{tikzpicture}[scale=.3]

\begin{scope}[xshift=-4cm,rotate=-45]
    \draw[thick, dotted] (0,0) circle (1 cm);
    \draw[thick] (1,0) .. controls (0,0) .. (0,1);
    \draw[thick] (-1,0) .. controls (0,0) .. (0,-1);
\end{scope}

\begin{scope}[rotate=45]
    \draw[thick, dotted] (0,0) circle (1 cm);
    \draw[thick] (-1,0) -- (1,0);
    \draw[thick] (0, .4) -- (0,1);
    \draw[thick] (0, -.4) -- (0,-1);
\end{scope}

\begin{scope}[xshift=4cm,rotate=45]
    \draw[thick, dotted] (0,0) circle (1 cm);
    \draw[thick] (1,0) .. controls (0,0) .. (0,1);
    \draw[thick] (-1,0) .. controls (0,0) .. (0,-1);
\end{scope}

\draw[thick, ->] (-1.2,0) -- (-2.8,0);
\draw[thick, ->] (1.2,0) -- (2.8,0);

\node[anchor=south] at (-2,0) {$0$};
\node[anchor=south] at (2, 0) {$1$};

\end{tikzpicture}. The binary strings $v \in \{0,1\}^n$ represent vertices of an $n$-dimensional cube, with weight denoted by $|v| = \sum_i v_i$. The natural poset relation in $\{0,1\}^n$ is denoted by $\prec$: $u \prec v$ if and only if $u_i \leq v_i$ for all $i$. 
If for all $j$ in a subset $J \subset \{1,\ldots, n\}$, $u_j = 0$ and $v_j = 1$, and $u_i = v_i$ for all other $i$, then we write $u \prec_k v$, where $k = |J|$.

Let $D_v$ denote the diagram for the smoothing of $D$ where the $i$th crossing is resolved according to the component $v_i$ of $v$. Denote the set of planar circles in $D_v$ by $Z_v$. 

A chain complex $\mathcal{C}$ lies over the cube $\{0,1\}^n$  if, up to an overall  homological grading shift,  the chain group decomposes as $\mathcal{C}_i = \bigoplus_{\alpha \in \{0,1\}^n, |\alpha| = i}\mathcal{C}_{\alpha}$. In this case, we say $\mathcal{C}$ has a grading induced by the cube $\{0,1\}^n$.

\subsubsection{Khovanov chain complex $\CKh(L)$ over $\BF$}

Let $V$ denote the two-dimensional $\BF$-vector space generated by $\{\vv_\pm\}$ in (quantum) degrees $\pm 1$. 
At vertex $v \in \{0,1\}^n$, associate the vector space 
\[
    \CKh(D_v) = V^{\otimes |Z_v|}.
\]
Let $n_+$ (\resp $n_-$) denote the number of positive (\resp negative)  crossings in $D$. 
The Khovanov chain groups (i.e.\ vector spaces over $\BF$) are given by
\[
    \CKh(D) = \bigoplus_{v \in \{0,1\}^n} \CKh(D_v)[-n_-]\{n_+ -2n_-\}.
\]
Here square brackets indicate a shift in homological grading, and curly brackets indicate a shift in the quantum grading: let $\mathcal{C}^{h,q}$ denote the chain group (vector space) at homological grading $h$ and quantum grading $q$; then $\left ( \mathcal{C}[i]\{j\} \right)^{h,q} = \mathcal{C}^{h-i, q-j}$.
To wit, the chain group at homological grading $h$ is
\[
    \CKh^h(D) = \bigoplus_{v \in \{0,1\}^n, \ h = |v|-n_-} \CKh(D_v).
\]
The set of \emph{Khovanov generators} are the pure tensors in $\vv_\pm$ at each vertex. 

The differential splits along the edges of the cube. For $u \prec_1 v$, the differential associated to the edge $D_u \to D_v$ in the cube of resolutions is given by either the \emph{merge map} $m$ or the \emph{split map} $\Delta$, depending on whether $|Z_u| > |Z_v|$ or $|Z_u| < |Z_v|$, respectively. The \emph{active circles} in $Z_u$ and $Z_v$ are those that differ between $Z_u$ and $Z_v$. The others are \emph{passive circles}. The component of the differential $d_{u,v}$ associated to the edge $u \to v$ is given by the following maps, extended by the identity map on the passive circles:

\begin{minipage}{.49\textwidth}
    \begin{align*}
        m : V \otimes V &\to V \\
        \vv_+ \otimes \vv_+ &\mapsto \vv_+ \\
        \vv_+ \otimes \vv_-, \vv_- \otimes \vv_+ &\mapsto \vv_-\\
        \vv_-\otimes\vv_- &\mapsto 0
    \end{align*}
\end{minipage}
\begin{minipage}{.49\textwidth}
    \begin{align*}
        \Delta : V &\to V \otimes V \\
        \vv_+ &\mapsto \vv_+ \otimes \vv_- + \vv_- \otimes \vv_+ \\
        \vv_- &\mapsto \vv_- \otimes \vv_-\\
    \end{align*}
\end{minipage}

The differential on $\CKh(D)$ is defined as 
\[
    d_{\Kh} = \sum_{u \prec_1 v} d_{u,v}.
\]

\subsubsection{Bar-Natan chain complex $\CBN(L)$ over $\BF[Q]/(Q^2)$}
\label{sec:BN-complex-intro}

The Bar-Natan homology
is defined via the same construction as above, but with the following modifications. The quantum degree of $Q$ is $-2$. 

The chain group at vertex $v$ is given by 
\[
    \CBN(D_v) = \CKh(D_v) \otimes_\BF \BF[Q]/(Q^2).
\]
The underlying $\BF$-vector space of $\CBN(D)$ has distinguished generators given by the set of pure tensors in $\vv_{\pm}$ and $\{1,Q\}$. 

The Bar-Natan differential is given by merge and split maps,
    \begin{align*}
        m : V \otimes V \otimes \BF[Q]/(Q^2) &\to V \otimes \BF[Q]/(Q^2) \\
        \vv_+ \otimes \vv_+ &\mapsto \vv_+ \\
        \vv_+ \otimes \vv_-, \vv_- \otimes \vv_+ &\mapsto \vv_-\\
        \vv_- \otimes \vv_- &\mapsto {Q \cdot \vv_- }
    \end{align*}
and 
    \begin{align*}
        \Delta : V \otimes \BF[Q]/(Q^2) &\to V \otimes V \otimes \BF[Q]/(Q^2) \\
        \vv_+ &\mapsto \vv_+ \otimes \vv_- + \vv_- \otimes \vv_+ 
            + {Q \cdot \vv_+ \otimes \vv_+}\\
        \vv_- &\mapsto \vv_- \otimes \vv_- ,
    \end{align*} 
respectively.
As in Khovanov homology, the differential on $\CBN(D)$ is defined as 
\[
    d_{\BN} = \sum_{u \prec_1 v} d^{BN}_{u,v},
\]
where $d^{BN}_{u,v}$ denotes the differential associated to the edge $u\to v$. The homology of this chain complex is denoted $\BN$. 

\subsubsection{Reduced versions}

Let $(L,p)$ be a based link in $S^3$. A generic projection yields a based diagram $(D,p)$ where $p$ is away from the crossings. At any vertex $D_v$ in the cube of resolutions, there is a distinguished circle containing the base point $p$. 
Let $\CKhred(D)$ be the subcomplex of $\CKh(D)$ consisting of elements where the based circle is labeled $\vv_-$. 
Reduced Khovanov homology \cite{Khovanov-reduced} $\Khred(L,p)$ is the homology of the quantum-shifted complex $\CKhred(D,p)\{1\}$. 
Similarly, let $\CBNred(D,p)$ be the subcomplex of $\CBN(D)$ given by $\CKhred(D) \otimes \BF[Q]/(Q^2)$. Reduced Bar-Natan homology $\BNred(L,p)$ is the homology of the quantum-shifted complex $\CBNred(D,p)\{1\}$. 
The homologies $\Khred(L,p)$ and $\BNred(L,p)$ are based link isotopy class invariants and do not depend on the chosen diagram. 
In particular, if $K$ is a knot, then up to isomorphism, $\Khred(K,p)$ and $\BNred(K,p)$ do not depend on the placement of the base point.

\subsubsection{Mirroring and duals}

Following the conventions of \cite{OzsvathSzaboss,Bloom,Lin,LOTborderedss}, we identify the first page $(E_1,d_1)$ of our spectral sequence with a \emph{homologically graded} version of $\BNred$; that is, the components of $d_1$ lie over the opposite cube $(1 \to 0)^n$, where the homological grading is still computed using the same cube grading $|v| = \sum_i v_i$ on the vertices $v$.
The homology of this complex is then identified with $\BNred(\overline{L})$ using Equation \ref{eq:mirror-dual} below. Thus the second page of our spectral sequence abutting to $\HFIhat(\Sigma(L))$ is isomorphic to the $\BNred$ of the \emph{mirror} $\overline{L}$ of the given link $L$.

The Frobenius system $\F_{\BN}$ defining $\BN$ is given by the commutative ring $R = \BF[Q]/(Q^2)$ and Frobenius algebra 
    \[
        A = R[X]/(X^2-QX) = \BF[Q,X]/(Q^2, X^2-QX)
    \]
with multiplication $m$ and comultiplication $\Delta$ defined previously in this section, unit $i:R \into A$ defined by $1 \mapsto 1$, and counit $\epsilon: A \to R$ given by $X \mapsto 1, 1 \mapsto 0$. Here, $1,X$ correspond to $\vv_+, \vv_-$ from earlier in this section, respectively. 

One can check that $\F_{\BN} = (A, i, \epsilon, m, \Delta)$  is self-dual (see \cite{Kh-Frobext,AKhZh-equiv}): there is an $R$-linear isomorphism 
\begin{align*}
    \varphi: A^* &\to A \\
            1^* &\mapsto X \\
            X^* &\mapsto 1
\end{align*}
under which $\F_{\BN}^* = (A^*,\epsilon^*, i^*, \Delta^*, m^*)$ are sent to  $\F_{\BN} = (A,i, \epsilon, m, \Delta)$. 
Therefore 
\begin{equation}
    \label{eq:mirror-dual}
    \CBN(\overline{D}, \F) \cong \CBN(D,\F^*)^* \cong \CBN(D, \F)^*,
\end{equation} 
where the second isomorphism is induced by $\varphi$ (see \cite[Example 3]{Kh-Frobext}); this identification reverses quantum grading.
We can identify each $\CBN(D_v)^*$ with $\CBN(D_v)$, since these are freely generated $R$-modules and $R$ is a principal ideal domain. We also reverse the homological grading under this dualization. After all these identifications, we have that $\CBN(\overline{D}, \F)$ is isomorphic to the homologically graded version of $\CBN$ seen on $(E_1,d_1)$.

Finally, given a base point $p$ on $D$, the corresponding base point $\overline p$ on $\overline D$ yields a subcomplex that is identified with the subcomplex $\CBNred(D,p)$. The homologies $E_2$ and $\BNred(D,p)$ are isomorphic, up to overall grading shifts.

\subsection{Bordered Heegaard Floer homology}
Introduced by Lipshitz, Ozsv\'ath, and Thurston \cite{LOT, LOTbimodules}, bordered Floer homology is an extension of Ozsv\'ath and Szab\'o's Heegaard Floer homology \cite{OS} for 3-manifolds with boundary. 
We first briefly discuss Heegaard Floer homology.

\subsubsection{Heegaard Floer homology}

Let $Y$ be a closed, oriented 3-manifold. The Heegaard Floer homology group $\HFhat(Y)$ of $Y$  is computed from the data of a Heegaard diagram $\HD = (\Sigma, \alphas, \betas, z)$ for $Y$. Here, $\Sigma \subset Y$ is a closed, oriented surface embedded in $Y$ that splits $Y$ into two handlebodies, i.e.\ $Y=H_{\alpha}\cup_{\Sigma}H_{\beta}$. The \emph{$\alpha$-curves} $\alphas$ (\resp \emph{$\beta$-curves} $\betas$) are a collection of mutually disjoint $g$ simple closed curves on $\Sigma$ such that each one of them bounds a disk in the handlebody $H_{\alpha}$ (resp. $H_{\beta}$). Here, $g$ denotes the genus of $\Sigma$. Further, $\Sigma\setminus \alphas$ and $\Sigma\setminus\betas$ is connected. Finally, $z\in\Sigma\setminus (\alphas\cup\betas)$ is a base point. The convention is to draw $\alpha$-curves in red and $\beta$-curves in blue.

The Heegaard Floer chain complex is the Lagrangian intersection Floer homology of the Lagrangian tori
\[
    \mathbb{T}_\alpha = \prod_{i=1}^g \alpha_i 
    \qquad \text{and} \qquad
    \mathbb{T}_\beta = \prod_{i=1}^g \beta_i
\]
inside the symplectic manifold 
\[
    \mathrm{Sym}^g(\Sigma) = \left ( \Sigma^{\times g} \right)/ S_g
\]
where the action of the symmetric group $S_g$ is given by permutation on the factors. 
Note that $\mathrm{Sym}^g(\Sigma)$ consists of unordered $g$-tuples of points on $\Sigma$. 
The chain complex $\CFhat(\HD)$ is generated by $g$-tuples of intersection points between $\alpha$ and $\beta$ curves. After fixing an almost complex structure, the differential counts pseudoholomorphic disks that miss the divisor of tuples containing the base point $z$. 

\subsubsection{Bordered 3-manifolds and associated bordered Floer modules}

A \emph{bordered 3-manifold} is a compact, oriented 3-manifold $Y$ with parametrized boundary $\partial Y$.  The parametrization is given by \emph{pointed matched circles} $\ZZ_i$, each of which encodes a handle decomposition of the the components $(\partial Y)_i$ of $\partial Y$ via homomorphisms $\phi_i: F(\ZZ_i) \to (\partial Y)_i$.

To a bordered 3-manifold with one parametrized boundary component, presented as a bordered Heegaard diagram $\HD$, Lipshitz, Ozsv\'ath, and Thurston \cite{LOT} construct an $\CA_\infty$-algebra $\CA(\ZZ)$ associated to the pointed matched circle $\ZZ$, and two types of modules: a left differential graded module $\CFDhat(\HD)$ over the algebra $\CA(-\ZZ)$ and a right $\CA_\infty$-module $\CFAhat(\HD)$ over $\CA(\ZZ)$. 
Moreover, in \cite{LOTbimodules}, the same authors show how to associate to a 3-manifold with two parametrized boundary components, presented as a bordered Heegaard diagram $\HD$ with two boundaries, the bimodules $\CFDDhat(\HD)$, $\CFDAhat(\HD)$, and $\CFAAhat(\HD)$. This technology can be extended to form trimodules, such as $\CFDAAhat(\HD)$, for 3-manifolds with three parametrized boundary components; such trimodules are used in the bordered reformulation \cite{LOTborderedss, LOTborderedss2} of the Ozsv\'ath-Szab\'o spectral sequence \cite{OzsvathSzaboss} and will be used in this paper. 

We will use the symbol $\CA$ to represent all $\CA_\infty$ algebras throughout; within context, it should be clear which pointed matched circle $\ZZ$ and surface $F(\ZZ)$ this algebra corresponds to. In particular, we only perform calculations with the torus algebra, which we discuss next.

\subsubsection{The torus algebra}
In this article we work often with surgery along knots, and as such will often work with the \emph{torus algebra}. This is the differential graded algebra $\CA = \CA(\ZZ)$ generated by the idempotents and chords 
\[
    \io0,\io1, \ro{1,2}, \ro{1,3}, \ro{1,4}, \ro{2,3},\ro{2,4}, \ro{3,4}
\]
with structure maps
\begin{align*}
\mu_2(\ro{i,j}, \ro{k,l}) &= 
    \begin{cases}
    \ro{i,l}    & \text{ if } j=k\\
    0           & \text{ if } j\neq k\\
    \end{cases}\\
\mu_2(\io{a},\ro{i,j}) &= 
    \begin{cases}
    \ro{i,j}    & \text{ if } i \equiv (a+1) \mod 2 \\
    0           & \text{ otherwise}
    \end{cases}\\
\mu_2(\ro{i,j},\io{a}) &= 
    \begin{cases}
    \ro{i,j}    & \text{ if } j \equiv (a+1) \mod 2 \\
    0           & \text{ otherwise}
    \end{cases}\\
\mu_n &= 0 \quad \text{for} \quad n \neq 2.
\end{align*}
That is, only $\mu_2$ is nontrivial, and it is given by concatenation. We will use the dot $\cdot$ to symbolize multiplication $\mu_2$ in $\CA$. Though we use the symbol ``$\iota$" for both idempotents in the torus algebra as well as the homotopy involution on the Heegaard Floer chain complex, within context it should be clear what the symbol stands for.

Moreover, for an algebra element $a \in \CA$, we denote the left and right idempotents of $a$ by $\iota_a$ and $_a\iota$, respectively. In other words, $\iota_a\cdot a=a\cdot\iota_a=a$.
In addition, $\iota_a^c$ denotes the complement of an idempotent $\iota_a$: $\io0^c = \io1$ and $\io1^c = \io0$. 

Finally, the dual over $\BF_2$ of $\CA(\ZZ)$ is denoted by $\ovl{\CA(\ZZ)}$. Specifically, it has 
\[
    \io0^*,\io1^*, \ro{1,2}^*, \ro{1,3}^*, \ro{1,4}^*, \ro{2,3}^*, \ro{2,4}^*, \ro{3,4}^*
\]
as a basis. 
Moreover, $\CA(\ZZ)$ acts on $\ovl{\CA(\ZZ)}$ from the left and the right as follows. For an $a^*\in\ovl{\CA(\ZZ)}$ and $b\in \CA(\ZZ)$, the elements $b\cdot a^*$ and $a^*\cdot b$ send $a'\in\CA(\ZZ)$ to $a^*(a'b)$ and $a^*(ba')$, respectively. Thus, the left idempotent $\iota_{a^*}$ of $a^*$ is equal to $_a\iota$, while the right idempotent of $_{a^*}\iota$ of $a^*$ is equal to $\iota_a$.

Over the next few sections we will review some of the bordered modules/bimodules over the torus algebra that we need to work with the involutive bordered Floer homology and compute surgery maps \cite{HendricksLipshitz}.

\subsubsection{Identity bimodules} \label{sec:identity}Let $\ZZ$ be the pointed matched circle for a torus. Then, the mapping cylinder of the identity map from $F(\ZZ)$ to $F(\ZZ)$ is denoted by $\Id_{\ZZ}$. By \cite[Theorem 4]{LOTbimodules}, this bimodule is homotopy equivalent to the \emph{identity  bimodule} $^{\CA(\ZZ)}[\mathrm{Id}_{\CA(\ZZ)}]_{\CA(\ZZ)}$ defined in \cite[Definition 2.2.48]{LOTbimodules} as follows. The bimodule $[\mathrm{Id}_{\CA(\ZZ)}]=\ ^{\CA(\ZZ)}[\mathrm{Id}_{\CA(\ZZ)}]_{\CA(\ZZ)}$ has two generators corresponding to the idempotents $\iota_0$ and $\iota_1$, and $\delta^1_k=0$ for $k\neq 2$. The nontrivial components of the differential $\delta^1_2$ are: 
\begin{align*}
    \iota_0 \otimes \{ \ro{1,2}, \ro{3,4}, \ro{1,4}\} 
    & \mapsto \{ \ro{1,2}, \ro{3,4}, \ro{1,4}\} \otimes \iota_1 \\
    \iota_1 \otimes \ro{2,3} 
    & \mapsto \ro{2,3} \otimes \iota_1\\
    \iota_0 \otimes \ro{1,3} 
    & \mapsto \ro{1,3} \otimes \iota_0\\
    \iota_1 \otimes \ro{2,4} 
    & \mapsto \ro{2,4} \otimes \iota_1
\end{align*}

\subsubsection{The Auroux-Zarev pieces}
\label{sec:AZpieces}
The next bordered bimodules we need for working with Hendricks and Lipshitz's involutive bordered Floer homology are $\CFDAhat(^\alpha\AZ(-\ZZ)^\beta)$ and $ \CFDAhat(^\beta\AZbar(\ZZ)^\alpha)$ corresponding to the Auroux-Zarev pieces $\AZ(-\ZZ)$ and $\AZbar(\ZZ)$, respectively, depicted in Figure \ref{fig:AZbar-AZ-bimod}. We will write the associated 
bimodule generators in terms of elements of the algebra $\CA(\ZZ)$, following
the notation in \cite[Section 2.4]{HendricksLipshitz}. 

First, we recall $\CFDAhat(^\alpha\AZ(-\ZZ)^\beta)$. 
The generators  are of the form $(\io{a})^c \otimes a$: 
\[
{
     \io0 \otimes \{ \io1, \ro{2,3}, \ro{2,4}\}
}
\qquad
{
    \io1 \otimes \{ \io0, \ro{1,2}, \ro{1,3}, \ro{1,4}, \ro{3,4}\}
}
\]
where $a\in \CA(\ZZ)$. See Figure \ref{fig:AZbar-AZ-bimod} for a labeling of the generators by the algebra elements. The map  $\delta_1^1$ is given by:
\begin{align*}
    \io1 \otimes \io0 
        &\mapsto \ro{2,3} \otimes \ro{2,3} \\
    \io1 \otimes \{ \ro{1,2},\ro{1,3},\ro{1,4}\}
        &\mapsto 0 \\
    \io1 \otimes \ro{3,4}
        &\mapsto \ro{2,3} \otimes  \ro{2,4} \\
    \io0 \otimes \io1
        &\mapsto \ro{1,2} \otimes  \ro{1,2} + \ro{3,4}\otimes \ro{3,4} + \ro{1,4} \otimes \ro{1,4}\\
    \io0 \otimes \ro{2,3}
        &\mapsto \ro{1,2} \otimes \ro{1,3} \\
    \io0 \otimes \ro{2,4} 
        &\mapsto \ro{1,2} \otimes \ro{1,4}
\end{align*}

The map $\delta_2^1$ is just right multiplication, i.e.
    \[ 
        (\iota_a \otimes a) \otimes a' \mapsto  \iota_a \otimes (a\cdot a'),
    \]
where $(\iota_a)^c \otimes a$ is a generator, $a' \in \CA$, and $a\cdot a'$ is the product of $a$ and $a'$ in the algebra, viewing $a$ as an algebra element.

Next, we recall the bimodule $ \CFDAhat(^\beta\AZbar(\ZZ)^\alpha)$.
The generators are of the form $(\iota_{a^*})^c\otimes a^*=\ _a\iota \otimes a^*$:
\[
{
    \io1 \otimes \{ \iod0, \rod{1,3}, \rod{2,3}  \}
}
\qquad
{
    \io0 \otimes \{ \iod1, \rod{1,2}, \rod{1,4}, \rod{2,4}, \rod{3,4}  \}
}
\]

The map $\delta_1^1$ is:
\begin{align*}
    \io1 \otimes \iod0 
        & \mapsto 0 \\
    \io1 \otimes \rod{1,3}
        & \mapsto \ro{2,3} \otimes \rod{1,2} \\
    \io1 \otimes \rod{2,3} 
        & \mapsto \ro{2,3} \otimes \iod1\\
    \io0 \otimes \iod1 
        & \mapsto 0 \\
    \io0 \otimes \rod{1,2}
        & \mapsto \ro{1,2} \otimes \iod0\\
    \io1 \otimes \rod{1,4}
        & \mapsto \ro{1,4} \otimes \iod0 + \ro{2,3} \otimes \rod{1,3}\\
    \io0 \otimes \rod{2,4}
        & \mapsto \ro{3,4} \otimes \rod{2,3} \\
    \io0 \otimes \rod{3,4} 
        & \mapsto \ro{3,4} \otimes \iod0\\
\end{align*}

The map $\delta_2^1$ is the right multiplication by $\CA(\ZZ)$, i.e.:
\[ 
    (_a\iota \otimes a^*) \otimes a' \mapsto\  _a\iota \otimes (a^* \cdot a').
\]
Explicitly, the only nonzero components are the following:
\begin{align*}
    (_a\iota \otimes a^*) \otimes a 
        &\mapsto\ _a\iota\otimes\,_a\iota^* \\    
    (\io1 \otimes \rod{1,3}) \otimes \ro{1,2}  
        &\mapsto 
        \io1 \otimes \rod{2,3}\\
    (\io0 \otimes \rod{2,4}) \otimes \ro{2,3}
        &\mapsto \io0 \otimes \rod{3,4} \\
    (\io0 \otimes \rod{1,4}) \otimes \ro{1,2}
        &\mapsto \io0 \otimes \rod{2,4} \\
    (\io0 \otimes \rod{1,4}) \otimes \ro{1,3}
        &\mapsto \io0 \otimes \rod{3,4}
\end{align*}
For all $k\geq 3$, $\delta_k^1 = 0$.

We now pair the two Auroux-Zarev bimodules together; we will need this when we explicitly compute the chain homotopy equivalence $\Omega$ between $[\mathrm{Id}_{\CA(\ZZ)}]$ and 
\[\CFDAhat(\,^\beta\AZbar(\ZZ)^{\alpha}) \DT \CFDAhat(\,^{\alpha}\AZ(-\ZZ)^{\beta})\]
in Section \ref{section-omega}.

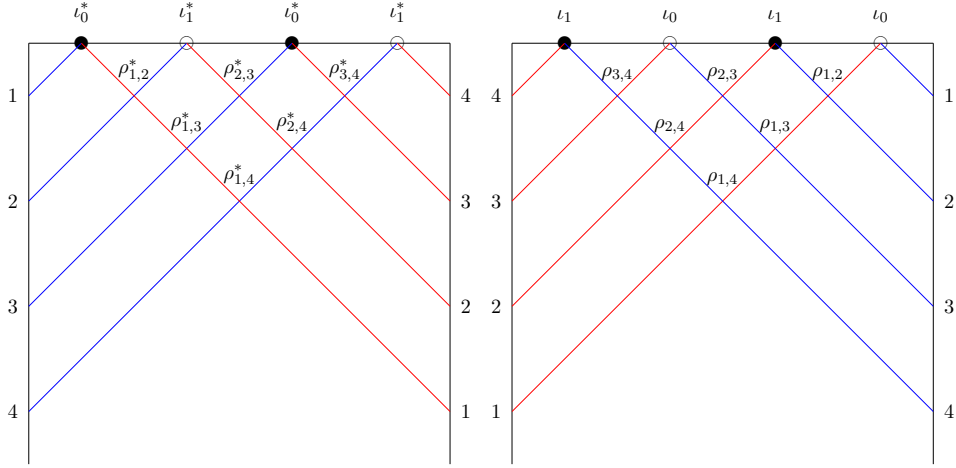
\begin{figure}
    \centering
    \scalebox{0.7}{
    \begin{tikzpicture}
\def\w{6} 
\def\h{8} 
\def\gap{.3} 
\draw (0,0) -- (0,\h);
\draw (\h,0) -- (\h,\h);
\draw (0,\h) -- (\h, \h);
\node[label={$\iota_0^*$}] (i1L) at (1,\h) {$\newmoon$};
\node[label={$\iota_1^*$}] (i0L) at (3,\h) {$\fullmoon$};
\node[label={$\iota_0^*$}] (i1R) at (5,\h) {$\newmoon$};
\node[label={$\iota_1^*$}] (i0R) at (7,\h) {$\fullmoon$};
\foreach \k in {1,2,3,4} {
    \node at (\h+\gap, -1+2*\k) {$\k$};
    \node at (0-\gap, \h+1-2*\k) {$\k$};
}
\foreach \k in {1,2,3,4} {
    \draw[blue] (0, -1+2*\k) -- (\h+1-2*\k,\h);
}
\foreach \k in {1,2,3,4} {
    \draw[red] (\h+1-2*\k,\h) -- (\h, \h+1-2*\k);
}
\node[label={$\rod{1,2}$}] at (2,7) {};
\node[label={$\rod{2,3}$}] at (4,7) {};
\node[label={$\rod{3,4}$}] at (6,7) {};
\node[label={$\rod{1,3}$}] at (3,6) {};
\node[label={$\rod{2,4}$}] at (5,6) {};
\node[label={$\rod{1,4}$}] at (4,5) {};
\end{tikzpicture}
    \begin{tikzpicture}
\def\w{6} 
\def\h{8} 
\def\gap{.3} 
\draw (0,0) -- (0,\h);
\draw (\h,0) -- (\h,\h);
\draw (0,\h) -- (\h, \h);
\node[label={$\iota_1$}] (i1L) at (1,\h) {$\newmoon$};
\node[label={$\iota_0$}] (i0L) at (3,\h) {$\fullmoon$};
\node[label={$\iota_1$}] (i1R) at (5,\h) {$\newmoon$};
\node[label={$\iota_0$}] (i0R) at (7,\h) {$\fullmoon$};
\foreach \k in {1,2,3,4} {
    \node at (0-\gap, -1+2*\k) {$\k$};
    \node at (\h+\gap, \h+1-2*\k) {$\k$};
}
\foreach \k in {1,2,3,4} {
    \draw[red] (0, -1+2*\k) -- (\h+1-2*\k,\h);
}
\foreach \k in {1,2,3,4} {
    \draw[blue] (\h+1-2*\k,\h) -- (\h, \h+1-2*\k);
}
\node[label={$\ro{3,4}$}] at (2,7) {};
\node[label={$\ro{2,3}$}] at (4,7) {};
\node[label={$\ro{1,2}$}] at (6,7) {};
\node[label={$\ro{2,4}$}] at (3,6) {};
\node[label={$\ro{1,3}$}] at (5,6) {};
\node[label={$\ro{1,4}$}] at (4,5) {};
\end{tikzpicture}
    }
    \caption{The diagrams $^\beta\AZbar(\ZZ)^\alpha$ and $^\alpha\AZ(-\ZZ)^\beta$, about to be paired.}
    \label{fig:AZbar-AZ-bimod}
\end{figure}

To simplify the notation, we do not write down the idempotents for the generators of $\CFDAhat(^\beta\AZbar(\ZZ)^\alpha)$ and $\CFDAhat(^\alpha\AZ(-\ZZ)^\beta)$. For instance, we denote a generator $(\iota_a)^c\otimes a$ by $a$. With this notation, the generators are of the form $g^* \bt g'$, where $g^*\in\ovl{\CA(\ZZ)}$ and $g\in \CA(-\ZZ)$, such that
 $_{g^*}\iota$ is complementary to $\iota_{g'}$. Note that $_{g^*}\iota = \iota_g$.
Explicitly, the generators are
{
\[
        \{ \iod0, \rod{1,2}, \rod{1,3}, \rod{1,4}, \rod{3,4} \} 
        \bt \{ \io1, \ro{2,3}, \ro{2,4} \}
\]
}
and
{
\[
        \{ \iod1, \rod{2,3}, \rod{2,4} \}
        \bt \{ \io0, \ro{1,2}, \ro{1,3}, \ro{1,4}, \ro{3,4} \},
\]
}
with the appropriate left idempotents. (The left idempotent of $a^* \bt b$ is $(_a \iota)^c$. For example, $\io0 \otimes \rod{1,2} \bt \ro{2,4}$ is a generator.)

We split the computation of $\delta_1^1$ into two pieces. 
The first set of terms comes from rectangles in the diagram that do not hit the boundary. In other words, the coefficient of the image will be the algebra element 1. The nonzero terms of this type are as follows:
\begin{align*}
     \rod{1,2} \bt \io1 &\mapsto   \iod1 \bt \ro{1,2} \\
      \rod{2,3} \bt \io0 &\mapsto  \iod0 \bt \ro{2,3} \\
     \rod{3,4} \bt \io1 &\mapsto   \iod1 \bt \ro{3,4} \\
     \rod{1,2} \bt \ro{2,3} &\mapsto   \iod1 \bt \ro{1,3} \\
     \rod{1,2} \bt \ro{2,4} &\mapsto   \iod1 \bt \ro{1,4} \\
      \rod{2,3} \bt \ro{3,4} &\mapsto  \iod0 \bt \ro{2,4} \\
      \rod{1,3} \bt \io1 &\mapsto  \rod{2,3} \bt \ro{1,2} \\
     \rod{2,4} \bt \io0 &\mapsto   \rod{3,4} \bt \ro{2,3} \\
     \rod{1,4} \bt \io1 &\mapsto    \iod1 \bt \ro{1,4} +  \rod{2,4} \bt \ro{1,2} \\
      \rod{1,3} \bt \ro{2,3} &\mapsto  \rod{2,3} \bt \ro{1,3} \\
     \rod{2,4} \bt \ro{3,4} &\mapsto   \rod{3,4} \bt \ro{2,4} \\
     \rod{1,4} \bt \ro{2,4} &\mapsto   \rod{2,4} \bt \ro{1,4} \\
      \rod{1,3} \bt \ro{2,4} &\mapsto  \rod{2,3} \bt \ro{1,4} \\
     \rod{1,4} \bt \ro{2,3} &\mapsto   \rod{2,4} \bt \ro{1,3} 
\end{align*}

The second set of terms consists of components that do hit the left boundary. In these cases, the differential outputs a non-idempotent algebra element.
\begin{align*}
    \iod0 \bt \{ \io1, \ro{2,3}, \ro{2,4} \}
        &\mapsto 0 \\
    \rod{1,2} \bt \{\io1, \ro{2,3}, \ro{2,4} \}
        &\mapsto \ro{1,2} \otimes \iod0 \bt \{ \io1, \ro{2,3}, \ro{2,4} \} \\
      \rod{1,3} \bt \{\io1, \ro{2,3}, \ro{2,4} \}  
        &\mapsto (\ro{1,3} \otimes \iod0  + \ro{2,3} \otimes \rod{1,2} ) 
            \bt \{ \io1, \ro{2,3}, \ro{2,4} \} \\
     \rod{1,4} \bt \{\io1, \ro{2,3}, \ro{2,4} \}
        &\mapsto (\ro{1,4} \otimes \iod0 + \ro{2,4} \otimes \rod{1,2} + \ro{3,4} \otimes \rod{1,3})
            \bt \{ \io1, \ro{2,3}, \ro{2,4} \} \\
     \rod{3,4} \bt \{\io1, \ro{2,3}, \ro{2,4} \}
        &\mapsto \ro{3,4} \otimes \iod0 \bt \{ \io1, \ro{2,3}, \ro{2,4} \} \\
     \iod1 \bt \{ \io0, \ro{1,2}, \ro{1,3}, \ro{1,4}, \ro{3,4} \}
        &\mapsto 0 \\
      \rod{2,3} \bt \{ \io0, \ro{1,2}, \ro{1,3}, \ro{1,4}, \ro{3,4} \}
        &\mapsto \ro{2,3} \otimes \iod1 
            \bt  \{ \io0, \ro{1,2}, \ro{1,3}, \ro{1,4}, \ro{3,4} \} \\
     \rod{2,4} \bt  \{ \io0, \ro{1,2}, \ro{1,3}, \ro{1,4}, \ro{3,4} \}
        &\mapsto ( \ro{2,4} \otimes \iod1 + \ro{3,4} \otimes \rod{2,3}) 
            \bt  \{ \io0, \ro{1,2}, \ro{1,3}, \ro{1,4}, \ro{3,4} \}
\end{align*}

To compute $\delta_2^1$, we first note that the only terms come from rectangles that hit the right boundary only; this means $\delta_2^1$ will only output the algebra element 1 on the left. In fact, $\delta_2^1$ is given by right multiplication:
\[
    g^* \bt g' \otimes a \mapsto g^* \bt (g'\cdot a)
\]
(the left idempotents have been suppressed).
Note that above, $a$ must be a non-idempotent element of $\CA$. 
For $k\geq 3$, $\delta^1_k = 0$.

\subsection{Branched double cover spectral sequence}\label{sec:borderedss}

Given a link $L$ in $S^3$, Ozsv\'{a}th and Szab\'{o} construct a spectral sequence whose second page is isomorphic to the reduced Khovanov homology $\widetilde{\Kh}(\ovl{L})$ of the mirror image of $L$, and which converges to the Heegaard Floer homology $\HFhat(\Sigma(L))$ of the branched double cover of $L$ \cite{OzsvathSzaboss}. Both homology theories are taken with $\BF = \BF_2$ coefficients. To compute this spectral sequence, Lipshitz, Ozsv\'{a}th, and Thurston constructed a bordered Floer homology description of this spectral sequence \cite{LOTborderedss, LOTborderedss2}. In this section, we will review this bordered description. 

Let $D$ be a diagram for the link $L$ which is the plat closure of a braid with $2k$ strands and $n$ crossings. Cut $S^3$ into $n+2$ pieces by slicing along $n+1$ nested $2$-dimensional spheres, and decompose the diagram as
$D=B_0B_1\cdots B_{n}B_{n+1}$ where $B_0$ and $B_{n+1}$ are $k$ cups and caps in $B^3$, respectively. Each $B_i$ is an elementary braid of $2k$ strands in $S^2 \times [0,1]$ for any $1\le i\le n$. Figure \ref{fig:elembraid} depicts an example of an elementary braid of six strands.

\begin{figure} [ht!]
\begin{tikzpicture}[scale=.7]
\foreach \x in {1,4,5,6}{
    \draw[ultra thick] (\x,0) -- (\x,1);
}
\def \d{.25};
\draw[ultra thick] (3,0) -- (2,1);
\filldraw[white] (2.5,.5) circle (\d cm);
\draw[ultra thick] (2,0) -- (3,1);
\begin{scope}[xshift=8cm]
\foreach \x in {1,4,5,6}{
    \draw[ultra thick] (\x,0) -- (\x,1);
}
\def \s{.5};
\draw[ultra thick] (2,0) .. controls (2,\s) and (3,\s) .. (3,0);
\draw[ultra thick] (2,1) .. controls (2,1-\s) and (3,1-\s) .. (3,1);
\end{scope}
\end{tikzpicture}
\caption{Left: An elementary braid of six strands. Right: The anti-braid-like resolution of the braid.} 
\label{fig:elembraid}
\end{figure}
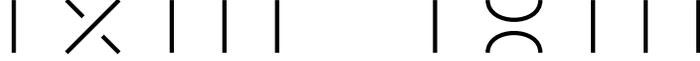

The branched double cover of each slice $S^2$ with $2k$ branch points $S^2\cap D$ is a surface $\Sigma$ of genus $k-1$. Further, the branched double covers $\Sigma(B_0)$ and $\Sigma(B_{n+1})$ are handlebodies with boundary $\Sigma$,  while $\Sigma(B_i)$ is the mapping cylinder of a positive or negative Dehn twist, depending on the type of crossing, along a homologically essential curve $\gamma_i\subset \Sigma$. Lipshitz, Ozsv\'{a}th, and Thurston choose a particular parametrization of $\Sigma$ by a \emph{linear} pointed match circle $\ZZ$ \cite[\S 5.1]{LOTborderedss} for which the mapping cylinder of the Dehn twist along $\gamma_i$ will have a convenient representation for all $i=1,\cdots, n$. This turns every $\Sigma(B_i)$ into a bordered $3$-manifold. Moreover, $\Sigma(B_0)$ and $\Sigma(B_{n+1})$ are \emph{plat} handlebodies in the sense of \cite[\S 5.2]{LOTborderedss}. By the pairing theorem for bordered Heegaard Floer homology \cite[\S 7]{LOT},

\begin{equation}\label{eq:pairing}
    \CFhat(\Sigma(D)) \simeq 
    \CFAhat(\Sigma(B_0))  
    \DT  \CFDAhat(\Sigma(B_1))  \DT  \cdots  
    \DT  \CFDAhat(\Sigma(B_n))
    \DT  \CFDhat(\Sigma(B_{n+1})).
\end{equation}
 
For $i=1,\cdots, n$, the crossing in $B_i$ can be resolved in two ways: a $0$-resolution denoted by $B_i^0$ and a $1$-resolution denoted by $B_i^1$. Depending on the type of crossing, one of the resolutions is \emph{braid-like} meaning that the result is the identity braid of $2k$ strands, while the other resolution is \emph{anti-braid-like}, as depicted in Figure~\ref{fig:elembraid}. For instance, suppose $\Sigma(B_i)$ is the mapping cylinder of the positive Dehn twist along $\gamma_i\subset \Sigma$. Then, $B_i^0$ and $B_i^1$ are braid-like and anti-braid-like, respectively. So, $\Sigma(B_i^0)=[0,1]\times \Sigma$, and $\Sigma(B_i^1)$ is obtained by performing $0$-surgery on $[0,1]\times \Sigma$ along $\{\frac{1}{2}\}\times\gamma_i$ equipped with the surface framing. A similar statement holds when $\Sigma(B_i)$ corresponds to the negative Dehn twist,  with the role of the $0$- and $1$-resolutions reversed.  By \cite[Theorem 2]{LOTborderedss}, $\CFDAhat(B_i)$ is homotopy equivalent to the mapping cone of a bimodule morphism $f_i$ from $\CFDAhat(B_i^1)$ to $\CFDAhat(B_i^0)$. 
 
Therefore, the right-hand side of Equation \ref{eq:pairing} is $\{0,1\}^n$-filtered, i.e.\ filtered with respect to the grading $|v|$ on the vertices of the cube. 
At each vertex $v=(v_1, \dots, v_n) \in \{0,1\}^n$, the pairing theorem in bordered Floer homology gives
\[
    \CFAhat(\Sigma(B_0))  \DT  \CFDAhat(\Sigma(B_1^{v_1}))  
    \DT  \cdots  
    \DT \CFDAhat(\Sigma(B_n^{v_n})) 
    \DT \CFDhat(\Sigma(B_{n+1})) \simeq\CFhat(\Sigma(D_v))
\]
where $D_v$ is the unlink obtained from $L$ by resolving all of the crossings as specified by $v$.
By construction, the induced spectral sequence converges to $\HFhat(\Sigma(D))$.
The first page $E^1$ of this spectral sequence is isomorphic to 
\[\bigoplus_{v\in\{0,1\}^n}\HFhat(\Sigma(D_v))\cong \bigoplus_{v\in\{0,1\}^n}\widetilde{\mathrm{CKh}}(D_v),\] 
where the isomorphism holds because the branched double cover $\Sigma(D_v)$ is a connected sum of copies of $S^1\times S^2$. Next, the differential on the $E_1$ page is given by $\mathrm{Id} \DT  f_i \DT \mathrm{Id}$ for $i=1,\cdots, n$, where the bimodule morphism $f_i$ will be described soon.  By \cite[Proposition 4.1]{LOTborderedss} and \cite[Theorem 6.3]{OzsvathSzaboss} these differentials are isomorphic to the differentials on the reduced Khovanov homology chain complex of $\ovl{D}$. Thus, the $E^2$ page is isomorphic to the reduced Khovanov homology $\widetilde{\mathrm{Kh}}(\ovl{L})$ of the mirror image of $L$.
 
In the rest of this section, we will briefly recall the construction of the bimodule morphism $f_i$. Let $\gamma$ be a simple closed curve in $\Sigma=F(\ZZ)$ and $Y_{\Id}$, $Y_{D_\gamma}$ and $Y_{D_{\gamma}^{-1}}$ be the bordered manifolds defined as the mapping cylinders of identity, positive, and negative Dehn twists along $\gamma\subset F(\ZZ)$, respectively. Thus, $Y_{\Id}=[0,1]\times F(\ZZ)$, while $Y_{D_{\gamma}}$ and $Y_{D_\gamma^{-1}}$ are obtained from $Y_{\Id}$ by performing $-1$ and $+1$ surgery (respectively) on $Y_\Id$ along the knot  $\{\frac{1}{2}\}\times\gamma$ equipped with surface framing induced from $\{\frac{1}{2}\}\times\Sigma$. Denote by $Y_{0(\gamma)}$ the bordered $3$-manifold obtained from $Y_{\Id}$ by performing $0$-surgery along the knot $\{\frac{1}{2}\}\times\gamma$.

Remove a tubular neighborhood $\nu(\{\frac{1}{2}\}\times\gamma)$ from $Y_{\Id}$, and equip the new torus boundary of $Y_{\Id}\setminus \nu(\{\frac{1}{2}\}\times\gamma)$ with a framing such that $Y_{D_{\gamma}}$, $Y_{0(\gamma)}$ and $Y_{\Id}$ are obtained by filling this torus boundary with slopes $0$, $\infty$, and $-1$, respectively. Connect the torus boundary of $Y_\Id \setminus \nu(\frac{1}{2} \times \gamma)$ via a framed arc to one of the other two boundary components, forming a triply bordered $3$-manifold; let $\HD$ be a bordered Heegaard diagram for it. Then $\HD\cup \HD_{\infty}$, $\HD\cup\HD_{-1}$, and $\HD\cup\HD_{0}$ are bordered Heegaard diagrams for $Y_{0(\gamma)}$, $Y_{\Id}$, and $Y_{D_{\gamma}}$, respectively. Here and throughout the paper, $\HD_{i}$ denotes the standard bordered Heegaard diagram for the solid torus with $i$-framing for $i= \infty, -1, 0$; see Figure \ref{fig:CFD-tori}.
   
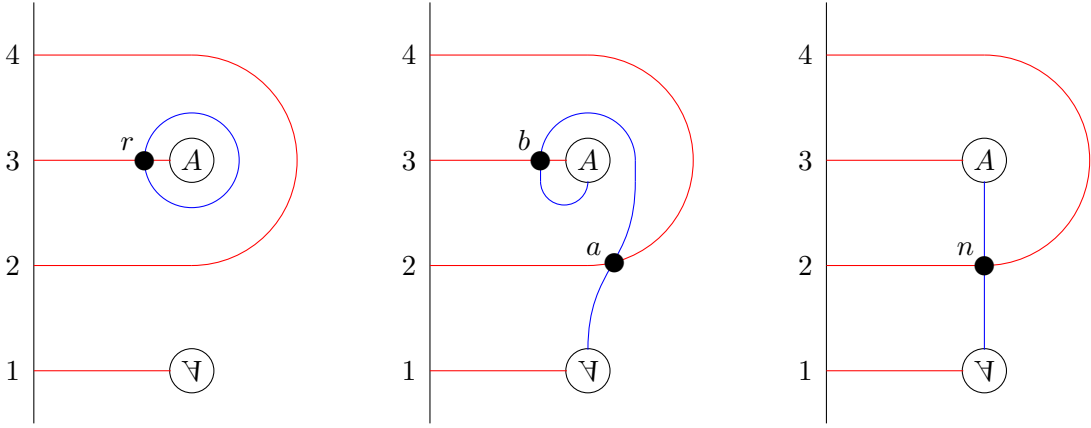
\begin{figure}
    \centering
    \begin{tikzpicture}[scale=.7]
\def\w{3} 
\def\r{.9} 
\def\gap{.4} 
%
\draw (0,1) -- (0,9);
\foreach \k in {1,2,3,4} {
        \node at (-\gap, 2*\k) {$\k$};
    }
\node[circle,inner sep=2pt,draw] (A1) at (\w,2) {\flip{$A$}};
\node[circle,inner sep=2pt,draw] (A2) at (\w,6) {$A$};
\foreach \y in {4,8} {
    \draw[red] (0,\y) -- (\w,\y);
    }
\foreach \y in {2,6} {
    \draw[red] (0,\y) -- (\w-\gap,\y);
    }
\draw[red] (\w,8) arc (90:-90:2); 
\draw[blue] (\w,6) circle (\r cm);
\node at (\w-\r,6) {$\newmoon$};
\node[anchor=south east] at (\w-\r,6) {$r$};
\end{tikzpicture}

\hspace{1cm}

\begin{tikzpicture}[scale=.7]
\def\w{3} 
\def\r{.9} 
\def\gap{.4} 
%
\draw (0,1) -- (0,9);
\foreach \k in {1,2,3,4} {
        \node at (-\gap, 2*\k) {$\k$};
    }
\node[circle,inner sep=2pt,draw] (A1) at (\w,2) {\flip{$A$}};
\node[circle,inner sep=2pt,draw] (A2) at (\w,6) {$A$};
\foreach \y in {4,8} {
    \draw[red] (0,\y) -- (\w,\y);
    }
\foreach \y in {2,6} {
    \draw[red] (0,\y) -- (\w-\gap,\y);
    }
\draw[red] (\w,8) arc (90:-90:2); 
\draw[blue] (\w+\r,6) arc (0:180:\r);
\draw[blue] (\w+\r,6) -- (\w+\r,6-\gap);
\draw[blue] (\w,6-\gap) arc (0:-180:0.5*\r);
\draw[blue](\w-\r,6-\gap) -- (\w-\r,6);
\draw[blue] (\w+\r,6-\gap) .. controls (\w+\r,4) and (\w,4)..  (\w,2+\gap);
\node at (\w-\r,6) {$\newmoon$};
\node[anchor=south east] at (\w-\r,6) {$b$};
\node at (\w+.5*\r+0.05,4+0.05) {$\newmoon$};
\node[anchor=south east] at (\w+.5*\r,4) {$a$};
\end{tikzpicture}

\hspace{1cm}

\begin{tikzpicture}[scale=.7]
\def\w{3} 
\def\r{.9} 
\def\gap{.4} 
%
\draw (0,1) -- (0,9);
\foreach \k in {1,2,3,4} {
        \node at (-\gap, 2*\k) {$\k$};
    }
\node[circle,inner sep=2pt,draw] (A1) at (\w,2) {\flip{$A$}};
\node[circle,inner sep=2pt,draw] (A2) at (\w,6) {$A$};
\foreach \y in {4,8} {
    \draw[red] (0,\y) -- (\w,\y);
    }
\foreach \y in {2,6} {
    \draw[red] (0,\y) -- (\w-\gap,\y);
    }
\draw[red] (\w,8) arc (90:-90:2); 
\draw[blue] (\w,6-\gap) -- (\w, 2+\gap);
\node at (\w,4) {$\newmoon$};
\node[anchor=south east] at (\w,4) {$n$};
\end{tikzpicture}
    \caption{From left to right, $\alpha$-bordered diagrams for the $\infty$-, $(-1)$-, and $0$-framed solid tori.}
    \label{fig:CFD-tori}
\end{figure}
   
There is a morphism $\phi$ from $\CFDhat(\HD_{\infty})$ to $\CFDhat(\HD_{-1})$ so that $\CFDhat(\HD_{0})$ is homotopy equivalent to  $\cone(\phi)$; see \cite[\S 11.2]{LOT}. 
As in \cite[Proof of Theorem 5]{LOTborderedss}, there is a homotopy equivalence
\begin{align*}
\CFDAhat(Y_{D_\gamma})&\simeq\CFDAhat(\HD\cup\HD_0)\\
&\simeq \CFDAAhat(\HD) \DT \CFDhat(\HD_0)\\
&\simeq\CFDAAhat(\HD) \DT \cone(\phi \colon \CFDhat(\HD_\infty)\to \CFDhat(\HD_{-1}))\\
&\simeq\cone\left(\mathrm{Id}_{\CFDAAhat(\HD)} \DT \phi \right)\\
&\simeq\cone\left(f_{\gamma}^+:\CFDAhat(Y_{0(\gamma)})\to\CFDAhat(Y_{\Id})\right).
\end{align*}
Similarly, $\CFDAhat(Y_{D_{\gamma}^{-1}})$ is homotopy equivalent to the mapping cone of a morphism $f_{\gamma}^-$ from $\CFDAhat(Y_{\Id})$ to $\CFDAhat(Y_{0(\gamma)})$. We define $f_i$  to be equal to $f_{\gamma_i}^+$ or $f_{\gamma_i}^-$, depending on the type of crossing in $B_i$.

We will apply this construction to the branched double covers $\Sigma(B_i)$, for $i = 1, \dots, n$, which are each mapping cylinders of Dehn twists along curves $\gamma_i$. 

\subsection{Involutive Heegaard Floer homology}

\subsubsection{Construction}
In involutive Heegaard Floer homology \cite{HendricksManolescu}, Hendricks and Manolescu construct the involutive chain map $\iota$ as follows. Start with a Heegaard diagram $\HD = (\Sigma, {\alpha}, {\beta}, z)$ for a closed, oriented 3-manifold $Y$ (we will focus on the case $Y = \Sigma(L)$). Let $\ovl{\ovl{\HD}} = (-\Sigma, {\beta}, {\alpha}, z)$ be the Heegaard diagram obtained from $\HD$ by reversing the roles of the $\alpha$ and $\beta$ curves, and reversing the orientation of the Heegaard surface. There is a natural isomorphism $\eta_\HD: \CFhat({\HD}) \to \CFhat(\ovl{\ovl{\HD}})$. Since ${\HD}$ and $\ovl{\ovl{\HD}}$ represent the same 3-manifold, there is a sequence of Heegaard moves that transform the diagram $\ovl{\ovl{\HD}}$ into $\HD$.  This sequence of moves induces a chain homotopy equivalence $\Phi_\HD: \CFhat(\ovl{\ovl{\HD}}) \to \CFhat({\HD})$ \cite{OS} which is well-defined up to chain homotopy \cite{OSfour, JuhaszThurston, HendricksManolescu}. The map $\iota: \CFhat({\HD}) \to \CFhat({\HD})$ defined by $\iota = \Phi_\HD \circ \eta_\HD$  is an involution up to homotopy, i.e.\ $\iota^2 \simeq \text{id}$. 

The involutive Heegaard Floer complex associated to $\HD$ is defined as 
\[
    \CFIhat(\HD)= \cone \left (
        \CFhat(\HD) \map{Q(id + \iota)} Q\cdot\CFhat(\HD)
    \right )
\]
and its homology is $\HFIhat(Y) = H_*(\CFIhat(\HD))$.

\begin{proposition}\label{prop:vertexmod}
Let $Y \cong \#^k (S^1 \times S^2)$. Then $\HFIhat(Y)$ is a rank one, free module over the ring 
\[ \Lambda^*H_1(Y) \otimes \BF[Q]/(Q^2) \]
generated by some class $\theta \in \HFIhat(Y).$
\end{proposition}
\begin{proof}
Let $\HD$ be the standard genus $1$ Heegaard diagram for $S^1\times S^2$ with one $\alpha$-circle and one $\beta$-circle so that $\beta$ is obtained from $\alpha$ by a small Hamiltonian isotopy and intersects it transversely in two points. Moreover, the base point $z$ is away from the support of the isotopy. Then, $\CFhat(\HD)$ has two generators in different Maslov gradings, and since $\iota$ is grading preserving, $\iota=\mathrm{Id}$ on $\CFhat(\HD)$. Thus, $\iota$ is equal to the identity on $\CFhat(\#^k\HD)$ as well \cite[Prop.~5.1]{HMZ-connected-sum}.
As a result,
\[\HFIhat(Y)\cong \HFhat(Y)\otimes \BF[Q]/(Q^2).\]
By the computations in \cite[\S 3.1]{OS-3m2} $\HFhat(Y)$ is a rank one free module over $\Lambda^*H_1(Y)$ generated by the homology class with largest Maslov grading.
\end{proof}

\subsubsection{Involutive surgery exact triangle}\label{sec:surgery}

This section is an overview of Hendricks and Lipshitz's surgery exact triangle for involutive bordered Floer homology from \cite[\S 7]{HendricksLipshitz}. Our explicit calculations appear later in \S \ref{sec:explicit-calcs}.

Let $K$ be a knot in a 3-manifold $Y$. Let $Y_\circ$ denote the $\circ$-surgery of $Y$ along $K$, for $\circ \in \{\infty, -1, 0\}$. Let $X = Y \wo \nu(K)$ (``exterior") denote the complement of a neighborhood of $K$ in $Y$. 
Let $\HD_X$ be a diagram for $X$ as a bordered manifold, and let $\HD_{\circ}$ be a diagram for the $\circ$-framed solid torus. 

Hendricks and Lipshitz construct a surgery exact triangle \cite[Theorem 7.1]{HendricksLipshitz} for involutive Floer homology using bordered Floer homology. Specifically, they prove that under the identification of 
\[
    \CFhat(Y_{\circ}) = \CFhat(\HD_{X}\cup\HD_{\circ})\simeq\CFAhat(\HD_{X})\DT\CFDhat(\HD_{\circ})
\]
using the pairing theorem, $\CFIhat(Y_0)$ is homotopy equivalent to the mapping cone of the chain map $\begin{bmatrix}\mathrm{Id}^2\DT\phi&0\\QG'&\mathrm{Id}^2\DT\phi\end{bmatrix}$ as in Figure \ref{fig:involutive-SM}.
\begin{figure}
\includestandalone{images/involutive-SM}
\caption{The bordered involutive short exact sequence.}
\label{fig:involutive-SM}
\end{figure}
The homotopy $G'$ is defined as the composition of the maps in Figure \ref{fig:involutive-surgery-diagram}. The components of the composition are as follows:
\begin{itemize}
    \item The chain map $\eta$ is the obvious one induced by reversing the roles of the $\alpha$- and $\beta$-curves (and arcs), and reversing the orientation on the Heegaard surfaces.
    \item 
    The chain homotopy equivalence $\Omega$ is a composition of the following identifications and chain homotopy equivalences:
    \begin{align*}
    \CFAhat(\HD_X) \DT \CFDhat(\HD_\infty)
        &= \CFAhat(\HD_X) \DT [\mathrm{Id}_{\CA(\ZZ)}] \DT \CFDhat(\HD_\infty) \\
        &\map{\Omega_1} \CFAhat(\HD_X) \DT \CFDAhat(\Id_\ZZ) \DT \CFDhat(\HD_\infty) \\
        &\map{\Omega_2} \CFAhat(\HD_X) \DT \CFDAhat(\ovl{\AZ})\DT\CFDAhat(\AZ) \DT \CFDhat(\HD_\infty)
    \end{align*}
Here, $\Omega_1$ is induced by the homotopy equivalence between $[\mathrm{Id}_{\CA(\ZZ)}]$ and $\CFDAhat(\Id_{\ZZ})$. Moreover, $\ovl{\AZ}\cup\AZ$ is a $\beta$-$\beta$ bordered Heegaard diagram for $\Id_{\ZZ}$. Thus, 
\[\CFDAhat(\Id_{\ZZ})\simeq \CFDAhat(\ovl{\AZ})\DT\CFDAhat(\AZ)\]
and $\Omega_2$ is the homotopy equivalence induced by the Heegaard moves. Consequently, $\Omega=\Omega_2\circ\Omega_1$ is induced by this homotopy equivalence between $[\mathrm{Id}_{\CA(\ZZ)}]$ and $\CFDAhat(\ovl{\AZ})\DT\CFDAhat(\AZ)$. Note that by \cite[Corollary 4.6]{HendricksLipshitz} such homotopy equivalence is unique up to homotopy and we will compute an explicit representative for it in Section \ref{section-omega}. Abusing the notation, ths homotopy equivalence between $[\mathrm{Id}_{\CA(\ZZ)}]$ and $\CFDAhat(\ovl{\AZ})\DT\CFDAhat(\AZ)$ is denoted by $\Omega$ as well.

\item $\ovl{\ovl{\HD_X}}\cup\ovl{\AZ}$ and $\HD_X$ are both bordered Heegaard diagrams for $X$ and the homotopy equivalence $\h$ from $\CFAhat(\ovl{\ovl{\HD_X}})\DT\CFDAhat(\ovl{\AZ})$ to $\CFAhat(\HD_X)$ is induced by the Heegaard moves.

\item In the proof of \cite[Lemma 7.2]{HendricksLipshitz} Hendricks and Lipshitz explicitly define the map $G$ from $\CFDAhat(\AZ)\DT\CFDhat(\ovl{\ovl{\HD_{\infty}}})$ to $\CFDhat(\HD_{-1})$, as follows. Bordered Heegaard diagrams $\HD_{\infty}$ and $\HD_{-1}$ are depicted in Figure \ref{fig:CFD-tori}. Using these diagrams, it is easy to see that  $\CFDhat(\ovl{\ovl{\HD_{\infty}}})$ has one generator $\ovl{\ovl{r}}$ and $\delta^1(\ovl{\ovl{r}})=\rho_{2,4}\otimes\ovl{\ovl{r}}$. Moreover, $\CFDhat(\HD_{-1})$ has two generators $a$ and $b$ and 
\[\delta^1(a)=(\rho_{1,2}+\rho_{3,4})\otimes b,\quad\quad\quad \delta^1(b)=0.\]
So, the generators of $\CFDAhat(\AZ)\DT\CFDhat(\ovl{\ovl{\HD_{\infty}}})$ are 
\[\{\iota_1,\rho_{1,2},\rho_{1,4}, \rho_{2,4}, \rho_{3,4} \}\bt \ovl{\ovl{r}}\]
and the nonzero terms of the differential $\delta^1$ are
\begin{align*}
    \io1\bt\ovl{\ovl{r}} & \mapsto \rho_{2,4}\bt \ovl{\ovl{r}}+ \rho_{1,2}\otimes (\rho_{1,2}\bt \ovl{\ovl{r}})+ \rho_{1,4}\otimes (\rho_{1,4}\bt \ovl{\ovl{r}})+\rho_{3,4}\otimes (\rho_{3,4}\bt \ovl{\ovl{r}})\\
    \rho_{1,2}\bt \ovl{\ovl{r}} & \mapsto \rho_{1,4}\bt \ovl{\ovl{r}}\\
    \rho_{2,4}\bt \ovl{\ovl{r}}& \mapsto \rho_{1,2}\otimes (\rho_{1,4}\bt \ovl{\ovl{r}})\\
    \rho_{3,4}\bt \ovl{\ovl{r}}& \mapsto \rho_{2,3}\otimes (\rho_{2,4}\bt \ovl{\ovl{r}}).
\end{align*}
Finally, the nonzero terms in the map $G$ are:
\begin{align*}
\rho_{2,4}\bt \ovl{\ovl{r}}&\mapsto a\\
\rho_{1,4}\bt \ovl{\ovl{r}}&\mapsto b\\
\rho_{1,2}\bt \ovl{\ovl{r}}&\mapsto \rho_{2,4}\otimes b
\end{align*}
\end{itemize}

\begin{figure}[h]
\begin{tikzpicture}[every text node part/.style={align=center}]
\def\dx{6}
\def\dy{2}
\node (04) at (0*\dx,4*\dy) {$
       \CFAhat(\HD_{X})
       \DT 
        \CFDhat(\HD_\infty)
         $};
\node (03) at (0*\dx,3*\dy) {$
       \CFAhat(\doublebar{\HD_{X}})
       \DT 
        \CFDhat(\doublebar{\HD_\infty})
    $};
\node (02) at (0*\dx,2*\dy) {$
       \CFAhat(\doublebar{\HD_{X}})
       \DT 
       \CFDAhat(\ovl{\AZ})
       $ \\ $
       \DT
       \CFDAhat(\AZ)
       \DT
        \CFDhat(\doublebar{\HD_\infty})
    $};
\node (01) at (0*\dx,1*\dy) {$
       \CFAhat(\HD_{X})
       \DT
        $ \\ $
        \CFDAhat(\AZ)
       \DT
        \CFDhat(\doublebar{\HD_\infty})
    $};
\node (10) at (1*\dx,0*\dy) {$
       \CFAhat(\HD_{X})
       \DT 
        \CFDhat(\HD_{-1})
    $};
\draw[->] (04) -- (03) node[midway,left] {{\footnotesize $\eta$}};
\draw[->] (03) -- (02) node[midway,left] {{\footnotesize $\Omega$}};
\draw[->] (02) -- (01) node[midway,left] {{\footnotesize $\h \DT \idmap^2$}};
\draw[->, dashed] (01) -- (10) node[pos=.4,above right] {{\footnotesize $\idmap \DT G$}};

\end{tikzpicture}
\caption{The composition of maps giving the homotopy $G'$.}
\label{fig:involutive-surgery-diagram}
\end{figure}

\section{Explicit computation of surgery maps}\label{sec:explicit-calcs}

We now compute the involutive surgery map for two specific knots in $\#^kS^1\times S^2$: the unknot, and the knot $S^1\times \{\mathit{pt}\}$ in one of the $S^1\times S^2$ summands. 
Specifically, we will prove the following proposition. 

\begin{proposition}\label{prop:surgerycomp}
Suppose $Y=\#^kS^1\times S^2$ and $K\subset Y$ is a knot specified by $S^1\times \{pt\}$ in one of the $S^1\times S^2$ summands. 
Let the framing induced by the product structure on $K$ be the $(-1)$-framing on $K$.
Then, $Y_{-1}(K)\cong \#^{k-1}(S^1\times S^2)$ and $H_1(Y_{-1}(K))\cong H_1(Y)/[K]$. Under the identification of $\HFIhat(Y)$ with $\Lambda^*H_1(Y) \otimes \mathbb{F}[Q]/(Q^2)$ under Proposition \ref{prop:vertexmod}, the surgery map \cite[Theorem 7.1]{HendricksLipshitz}
\[f: \HFIhat(Y)\to\HFIhat(Y_{-1}(K))\]
is the $\mathbb{F}[Q]/(Q^2)$-module homomorphism induced by the composition of the quotient map $q \colon H_1(Y) \rightarrow H_1(Y)/[K]$ followed by the isomorphism $H_1(Y)/[K]\cong H_1(Y_{-1}(K))$.

Dually, suppose $U\subset Y$ is the unknot, and it is framed so that $Y_{-1}(U)\cong \#^{k+1}(S^1\times S^2)$. Then, $H_1(Y)$ naturally injects into $H_1(Y_{-1}(U))$. Let $K'\subset Y_{-1}(U)$ be the core of the attached solid torus. Under the identification of Proposition \ref{prop:vertexmod} the surgery map \cite[Theorem 7.1]{HendricksLipshitz}
\[f': \HFIhat(Y)\to\HFIhat(Y_{-1}(U))\]
is given by $f'(\xi)=i(\xi)\wedge [K']+Qi(\xi)$ where the map $i$ is the inclusion of $H_1(Y)$ in $H_1(Y_{-1}(U))$. 
\end{proposition}

The rest of this section is dedicated to proving this proposition. We begin by computing explicitly certain components of Hendricks and Lipshitz's homotopies.

\subsection{Computing $\Omega$}
\label{section-omega} 
Let $\ZZ$ be the standard pointed matched circle for the torus, and $[\mathrm{Id}_{\CA(\ZZ)}]$ be the identity bimodule (Section \ref{sec:identity}). In this section, we will explicitly compute a representative for the homotopy equivalence 
\[
    \Omega : [\mathrm{Id}_{\ZZ}] \to \CFDAhat(\AZbar(\ZZ)) \DT \ \CFDAhat(\AZ(-\ZZ)).
\]

Note that the codomain of $\Omega$ is calculated in Section \ref{sec:AZpieces}. By definition $\Omega$ is a collection of maps 
\[
    \Omega^1_{1+j}: [\mathrm{Id}_{\ZZ}] \otimes \CA^j \to \CA \otimes \left (\CFDAhat(\AZbar) \DT \CFDAhat(\AZ)\right )
\]
for $j \geq 0$. Since $\Omega$ is a chain map, $d\Omega=0$, where $d\Omega$ is defined using the following diagrams: 
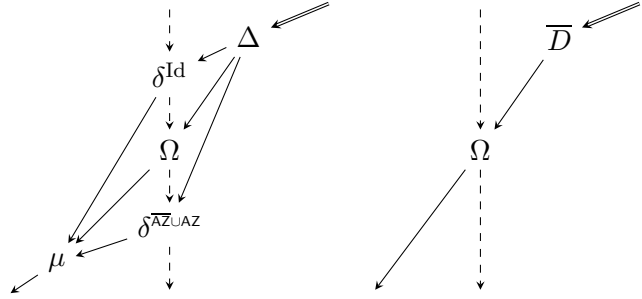
\begin{center}
    \scalebox{1}{
    \begin{tikzpicture}[xscale=1.5]
\node (partial) at (0,1) {$\delta^{\mathrm{Id}}$};
\node (Omega) at (0,0) {$\Omega$};
\node (delta) at (0,-1) {$\deltaAZAZ$};
\node (mu) at (-1,-1.5) {$\mu$};
\node (splitter) at (.7,1.5) {$\Delta$};
\node (in-alg) at (1.5,2) {};
\node (out-alg) at (-1.5,-2) {};
\node (input) at (0, 2) {};
\node (output) at (0,-2) {};
\draw[-stealth, dashed] (input) -- (partial);
\draw[-stealth, dashed] (partial) -- (Omega);
\draw[-stealth, dashed] (Omega) -- (delta);
\draw[-stealth, dashed] (delta) -- (output);
\draw[double, -stealth] (in-alg) -- (splitter);
\draw[-stealth] (mu) -- (out-alg);
\draw[-stealth] (splitter) -- (partial);
\draw[-stealth] (splitter) -- (Omega);
\draw[-stealth] (splitter) -- (delta);
%
\draw[-stealth] (partial) -- (mu);
\draw[-stealth] (Omega) -- (mu);
\draw[-stealth] (delta) -- (mu);
\end{tikzpicture}
\begin{tikzpicture}[xscale=1.5]
\node (Omega) at (0,0) {$\Omega$};
\node (Dbar) at (.7,1.5) {$\ovl D$};
\node (in-alg) at (1.5,2) {};
\node (out-alg) at (-1,-2) {};
\node (input) at (0, 2) {};
\node (output) at (0,-2) {};
\draw[-stealth, dashed] (input) -- (Omega);
\draw[-stealth, dashed] (Omega) -- (output);
\draw[double,stealth-] (Dbar) -- (in-alg);
\draw[-stealth] (Dbar) -- (Omega);
\draw[-stealth] (Omega) -- (out-alg);
\end{tikzpicture}
    }
\end{center}

Here, $\Delta$ is the splitting operator that allocates the input algebra elements, and $\ovl D$ is the differential in the standard bar complex, e.g.\ $\ovl D(a_1\otimes  a_2 \otimes a_3) = (a_1 \cdot a_2) \otimes a_3 + a_1 \otimes (a_2 \cdot a_3)$ (see \cite[\S 2.2.4] {LOTbimodules}). Furthermore,  being a homotopy equivalence between two $DA$ bimodules is equivalent to being a quasi-isomorphism by \cite[Corollary 2.4.4]{LOTbimodules}. Thus, we define $\Omega^1_1$ to be 
\[
\io0
    \mapsto  \iod1 \bt \io0 
\qquad \qquad
\io1
    \mapsto    \iod0 \bt \io1.
\]
In the remainder of this section, we define $\Omega^1_{1+j}$ for $j>0$ such that $d\Omega=0$. All large calculations were computer-assisted \cite{math-code}. Note that the following observations simplify our computations greatly:
\begin{itemize}
     \item  Since the torus algebra $\CA$ has $\mu_i = 0$ for all $i\neq 2$, we only need to check diagrams where the number of inputs to $\mu$ is 2.
    \item As discussed in \S \ref{sec:identity} the only nontrivial component of $\delta^{\mathrm{Id}}$ is $\delta^1_2$.
    \item By the above computations, the only nontrivial components of $\delta^{\scaleto{\AZbar\cup\AZ}{5pt}}$ are $\delta^1_1$ and $\delta^1_2$.
\end{itemize}

The only term contributing to $(d\Omega)^1_1$ is:
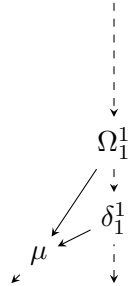
\begin{center}
    \scalebox{1}{
    \begin{tikzpicture}
\node (Omega) at (0,0) {$\Omega_1^1$};
\node (delta) at (0,-1) {$\delta^1_1$};
\node (mu) at (-1,-1.5) {$\mu$};
\node (in-alg) at (1.5,2) {};
\node (out-alg) at (-1.5,-2) {};
\node (input) at (0, 2) {};
\node (output) at (0,-2) {};
\draw[-stealth,dashed] (input) edge (Omega);
\draw[-stealth,dashed] (Omega) edge (delta);
\draw[stealth-,dashed] (output) edge (delta);
%
\draw[-stealth] (mu) edge (out-alg);
%
%
\draw[stealth-] (mu) edge (Omega);
\draw[stealth-] (mu) edge (delta);
\end{tikzpicture}
    }
\end{center}
It is straightforward that $\delta_1^1$ vanishes on the image of $\Omega_1^1$, so we have $(d\Omega)^1_1=0$. 

Next, there are three types of terms contributing to $(d\Omega)^1_2$:
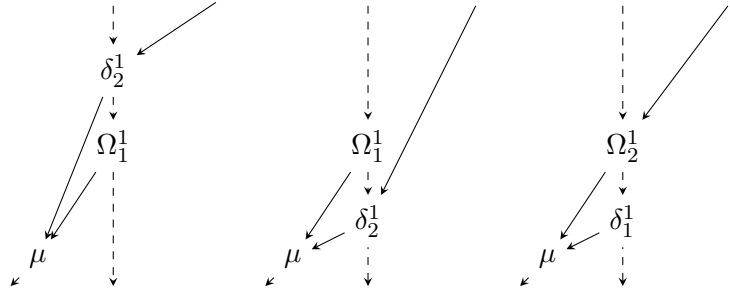
\begin{center}
    \scalebox{1}{

\begin{tikzpicture}
\node (partial) at (0,1) {$\delta_2^1$};
\node (Omega) at (0,0) {$\Omega_1^1$};
\node (mu) at (-1,-1.5) {$\mu$};
\node (in-alg) at (1.5,2) {};
\node (out-alg) at (-1.5,-2) {};
\node (input) at (0, 2) {};
\node (output) at (0,-2) {};
\draw[-stealth] (in-alg) edge (partial);
\draw[-stealth, dashed] (input) edge (partial);
\draw[-stealth, dashed ] (partial) edge (Omega);
\draw[-stealth, dashed] (Omega) edge (output);
%
\draw[-stealth] (mu) edge (out-alg);
%
%
\draw[stealth-] (mu) edge (partial);
\draw[stealth-] (mu) edge (Omega);
\end{tikzpicture}


\begin{tikzpicture}
\node (Omega) at (0,0) {$\Omega_1^1$};
\node (delta) at (0,-1) {$\delta_2^1$};
\node (mu) at (-1,-1.5) {$\mu$};
\node (in-alg) at (1.5,2) {};
\node (out-alg) at (-1.5,-2) {};
\node (input) at (0, 2) {};
\node (output) at (0,-2) {};
\draw[-stealth, dashed] (input) edge (Omega);
\draw[-stealth, dashed] (Omega) edge (delta);
\draw[stealth-, dashed] (output) edge (delta);
\draw[-stealth] (in-alg) edge (delta);
%
\draw[-stealth] (mu) edge (out-alg);
%
%
\draw[stealth-] (mu) edge (Omega);
\draw[stealth-] (mu) edge (delta);
\end{tikzpicture}


\begin{tikzpicture}
\node (Omega) at (0,0) {$\Omega_2^1$};
\node (delta) at (0,-1) {$\delta^1_1$};
\node (mu) at (-1,-1.5) {$\mu$};
\node (in-alg) at (1.5,2) {};
\node (out-alg) at (-1.5,-2) {};
\node (input) at (0, 2) {};
\node (output) at (0,-2) {};
\draw[-stealth, dashed] (input) edge (Omega);
\draw[-stealth, dashed] (Omega) edge (delta);
\draw[stealth-, dashed] (output) edge (delta);
%
\draw[-stealth] (mu) edge (out-alg);
\draw[-stealth] (in-alg) edge (Omega);
%
\draw[stealth-] (mu) edge (Omega);
\draw[stealth-] (mu) edge (delta);
\end{tikzpicture}
    }
\end{center}
We define $\Omega^1_2$ such that the third term is equal to the sum of the first two terms and so $(d\Omega)^1_2=0$. 
Let the non-zero terms in the definition of $\Omega^1_2$ be as follows: 
\begin{itemize}
    \item If $a$ is indecomposable (i.e.\ $a=\rho_{1,2}$, $\rho_{2,3}$ or $\rho_{3,4}$), then 
        \[ \Omega_2^1(\iota_a, a) =\  _a\iota \otimes a^* \bt \ _a\iota. \]
    \item If $a$ decomposes as $a = a_1a_2$ (i.e.\ $a=\rho_{1,3}$ or $\rho_{2,4}$), then 
    \[ \Omega_2^1(\iota_a, a) 
        = \ _{a_1}\iota \otimes a_1^* \bt a_2 + 
        a_1 \otimes a_2^* \bt\ _{a_2}\iota
        . \]
    \item If $a$ decomposes as $a = a_1 a_2 a_3$ (i.e.\ $a = \ro{1,4}$), then 
    \[ \Omega_2^1(\iota_a, a) 
        = \  _{a_1}\iota \otimes a_1^* \bt \ a_2a_3
        + a_1 \otimes a_2^* \bt a_3
        + a_1a_2 \otimes a_3^* \bt\ _{a_3}\iota
        . \]
\end{itemize}
(Note that when $a$ is an indecomposable strand, $\iota_a = (_a \iota)^c$.) It is straightforward to check that the above terms will sum up to zero and so $(d\Omega)_2^1=0$.  

Next, we define $\Omega^1_3$ so that $(d\Omega)^1_3=0$. There are four types of terms contributing to $(d\Omega)^1_3$:

\begin{center}
    \scalebox{1}
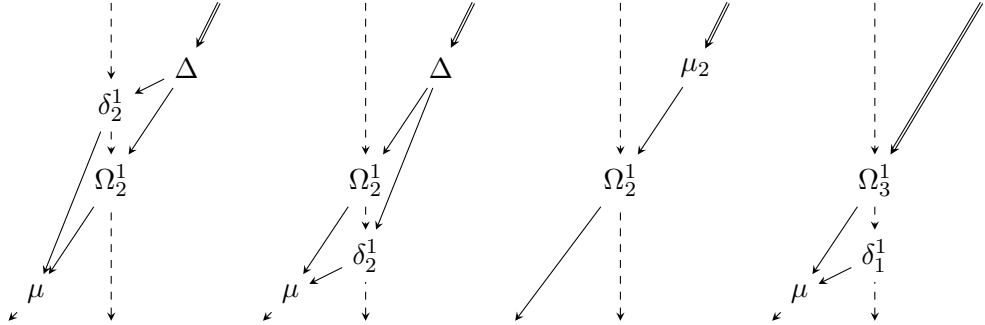
{
    \begin{tikzpicture}
\node (partial) at (0,1) {$\delta^1_2$};
\node (Omega) at (0,0) {$\Omega^1_2$};
\node (mu) at (-1,-1.5) {$\mu$};
\node (splitter) at (1,1.5) {$\Delta$};
\node (in-alg) at (1.5,2.5) {};
\node (out-alg) at (-1.5,-2) {};
\node (input) at (0, 2.5) {};
\node (output) at (0,-2) {};
\draw[-stealth, dashed] (input) edge (partial);
\draw[-stealth, dashed] (partial) edge (Omega);
\draw[stealth-, dashed] (output) edge (Omega);
\draw[double,stealth-] (splitter) -- (in-alg);
\draw[-stealth] (mu) edge (out-alg);
\draw[-stealth] (splitter) edge (partial);
\draw[-stealth] (splitter) edge (Omega);
%
\draw[stealth-] (mu) edge (partial);
\draw[stealth-] (mu) edge (Omega);
\end{tikzpicture}


\begin{tikzpicture}
\node (Omega) at (0,0) {$\Omega^1_2$};
\node (delta) at (0,-1) {$\delta^1_2$};
\node (mu) at (-1,-1.5) {$\mu$};
\node (splitter) at (1,1.5) {$\Delta$};
\node (in-alg) at (1.5,2.5) {};
\node (out-alg) at (-1.5,-2) {};
\node (input) at (0, 2.5) {};
\node (output) at (0,-2) {};
\draw[-stealth, dashed] (input) edge (Omega);
\draw[-stealth, dashed] (Omega) edge (delta);
\draw[stealth-, dashed] (output) edge (delta);
\draw[double, stealth-] (splitter) -- (in-alg);
\draw[-stealth] (mu) edge (out-alg);
%
\draw[-stealth] (splitter) edge (Omega);
\draw[-stealth] (splitter) edge (delta);
%
\draw[stealth-] (mu) edge (Omega);
\draw[stealth-] (mu) edge (delta);
\end{tikzpicture}


\begin{tikzpicture}
\node (Omega) at (0,0) {$\Omega_2^1$};
\node (splitter) at (1,1.5) {$\mu_2$};
\node (in-alg) at (1.5,2.5) {};
\node (out-alg) at (-1.5,-2) {};
\node (input) at (0, 2.5) {};
\node (output) at (0,-2) {};
\draw[-stealth] (splitter) edge (Omega);
\draw[double,-stealth] (in-alg) -- (splitter);
\draw[-stealth, dashed] (input) edge (Omega);
\draw[stealth-, dashed] (output) edge (Omega);
\draw[stealth-] (out-alg) edge (Omega);

\end{tikzpicture}


\begin{tikzpicture}
\node (Omega) at (0,0) {$\Omega^1_3$};
\node (delta) at (0,-1) {$\delta^1_1$};
\node (mu) at (-1,-1.5) {$\mu$};
\node (in-alg) at (1.5,2.5) {};
\node (out-alg) at (-1.5,-2) {};
\node (input) at (0, 2.5) {};
\node (output) at (0,-2) {};
\draw[-stealth, dashed] (input) edge (Omega);
\draw[-stealth, dashed] (Omega) edge (delta);
\draw[stealth-, dashed] (output) edge (delta);
\draw[double, stealth-] (Omega) -- (in-alg);
\draw[-stealth] (mu) edge (out-alg);
%
%
\draw[stealth-] (mu) edge (Omega);
\draw[stealth-] (mu) edge (delta);
\end{tikzpicture}
    }
\end{center}

We compute that $\Omega^1_3$ is:
\begin{align*}
    \io0 \otimes \ro{1,3} \otimes \ro{1,2} 
    &\mapsto \ro{1,2} \otimes \ro{1,3}^* \bt \io1\\
    \io0 \otimes \ro{1,3} \otimes \ro{1,3} 
    &\mapsto \ro{1,2} \otimes \ro{1,3}^* \bt \ro{2,3}\\
    \io0 \otimes \ro{1,3} \otimes \ro{1,4}
    &\mapsto \ro{1,2} \otimes \ro{1,3}^* \bt \ro{2,4}\\
    \io0 \otimes \ro{1,4} \otimes \ro{2,3}
    &\mapsto \ro{1,3} \otimes \ro{2,4}^* \bt \io0\\
    \io0 \otimes \ro{1,4} \otimes \ro{2,4}
    &\mapsto \ro{1,3} \otimes \ro{2,4}^* \bt \ro{3,4}\\
    \io0 \otimes \ro{3,4} \otimes \ro{2,3}
    &\mapsto  \ro{2,4}^* \bt \io0 \\
    \io0 \otimes \ro{3,4} \otimes \ro{2,4}
    &\mapsto  \ro{2,4}^* \bt \ro{3,4} \\
    \io1 \otimes \ro{2,3} \otimes \ro{1,2}
    &\mapsto   \ro{1,3}^* \bt \io1\\
    \io1 \otimes \ro{2,3} \otimes \ro{1,3}
    &\mapsto   \ro{1,3}^* \bt \ro{2,3}\\
    \io1 \otimes \ro{2,3} \otimes \ro{1,4}
    &\mapsto   \ro{1,3}^* \bt \ro{2,4}\\
    \io1 \otimes \ro{2,4} \otimes \ro{2,3}
    &\mapsto \ro{2,3} \otimes \ro{2,4}^* \bt \io0\\
    \io1 \otimes \ro{2,4} \otimes \ro{2,4}
    &\mapsto \ro{2,3} \otimes \ro{2,4}^* \bt \ro{3,4}
\end{align*}

Finally, there are five types of terms in $(d\Omega)_4^1$:
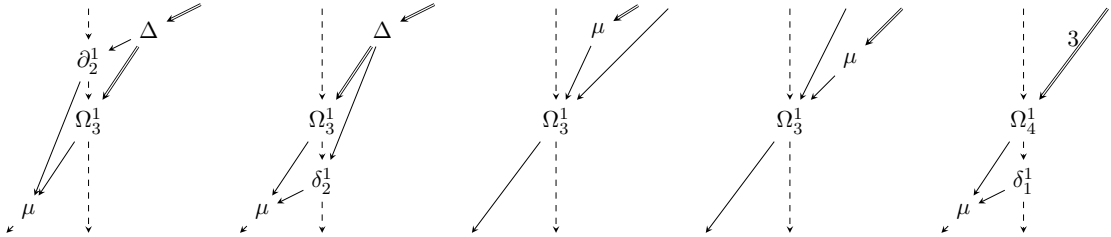
\begin{center}
    \scalebox{.8}{
    \begin{tikzpicture}
\node (partial) at (0,1) {$\partial^1_2$};
\node (Omega) at (0,0) {$\Omega^1_3$};
\node (mu) at (-1,-1.5) {$\mu$};
\node (splitter) at (1,1.5) {$\Delta$};
\node (in-alg) at (2,2) {};
\node (out-alg) at (-1.5,-2) {};
\node (input) at (0, 2) {};
\node (output) at (0,-2) {};
\draw[-stealth, dashed] (input) -- (partial);
\draw[-stealth, dashed] (partial) -- (Omega);
\draw[-stealth, dashed] (Omega) -- (output);
\draw[double, -stealth] (in-alg) -- (splitter);
\draw[-stealth] (mu) -- (out-alg);
\draw[-stealth] (splitter) -- (partial);
\draw[double, -stealth] (splitter) -- (Omega);
%
\draw[-stealth] (partial) -- (mu);
\draw[-stealth] (Omega) -- (mu);
\end{tikzpicture}
%
%
\begin{tikzpicture}
\node (Omega) at (0,0) {$\Omega^1_3$};
\node (delta) at (0,-1) {$\delta^1_2$};
\node (mu) at (-1,-1.5) {$\mu$};
\node (splitter) at (1,1.5) {$\Delta$};
\node (in-alg) at (2,2) {};
\node (out-alg) at (-1.5,-2) {};
\node (input) at (0, 2) {};
\node (output) at (0,-2) {};
\draw[-stealth,dashed] (input) -- (Omega);
\draw[-stealth,dashed] (Omega) -- (delta);
\draw[-stealth,dashed] (delta) -- (output);
\draw[double, -stealth] (in-alg) -- (splitter);
\draw[-stealth] (mu) -- (out-alg);
%
\draw[double, -stealth] (splitter) -- (Omega);
\draw[-stealth] (splitter) -- (delta);
%
\draw[-stealth] (Omega) -- (mu);
\draw[-stealth] (delta) -- (mu);
\end{tikzpicture}
%
%
\begin{tikzpicture}
\node (Omega) at (0,0) {$\Omega^1_3$};
\node (mu) at (.7,1.5) {$\mu$};
\node (in-alg1) at (1.5,2) {};
\node (in-alg2) at (2,2) {};
\node (out-alg) at (-1.5,-2) {};
\node (input) at (0, 2) {};
\node (output) at (0,-2) {};
\draw[-stealth, dashed] (input) -- (Omega);
\draw[-stealth, dashed] (Omega) -- (output);
\draw[double, -stealth] (in-alg1) -- (mu);
\draw[-stealth] (mu) -- (Omega);
\draw[-stealth] (in-alg2) -- (Omega);
\draw[-stealth] (Omega) -- (out-alg);
%
%
\end{tikzpicture}


\begin{tikzpicture}
\node (Omega) at (0,0) {$\Omega^1_3$};
\node (mu) at (1,1) {$\mu$};
\node (in-alg1) at (1,2) {};
\node (in-alg2) at (2,2) {};
\node (out-alg) at (-1.5,-2) {};
\node (input) at (0, 2) {};
\node (output) at (0,-2) {};
\draw[-stealth, dashed] (input) -- (Omega);
\draw[-stealth, dashed] (Omega) -- (output);
\draw[double, -stealth] (in-alg2) -- (mu);
\draw[-stealth] (mu) -- (Omega);
\draw[-stealth] (in-alg1) -- (Omega);
\draw[-stealth] (Omega) -- (out-alg);
%
%
\end{tikzpicture}
%
%
\begin{tikzpicture}
\node (Omega) at (0,0) {$\Omega^1_4$};
\node (delta) at (0,-1) {$\delta^1_1$};
\node (mu) at (-1,-1.5) {$\mu$};
\node (in-alg) at (1.5,2) {};
\node (out-alg) at (-1.5,-2) {};
\node (input) at (0, 2) {};
\node (output) at (0,-2) {};
\draw[-stealth, dashed] (input) -- (Omega);
\draw[-stealth, dashed] (Omega) -- (delta);
\draw[-stealth, dashed] (delta) -- (output);
\draw[double, -stealth] (in-alg) -- node[midway,above] {3} ++ (Omega);
\draw[-stealth] (mu) -- (out-alg);
%
%
\draw[-stealth] (Omega) -- (mu);
\draw[-stealth] (delta) -- (mu);
\end{tikzpicture}

    }
\end{center}
We compute that the first four terms sum to zero \cite{math-code}; thus it suffices to set $\Omega^1_4=0$. Similarly, it suffices to set $\Omega^1_k=0$ for all $k\ge 5$, as well. 

\subsection{Proof of Proposition \ref{prop:surgerycomp}}

We are now ready to prove Proposition \ref{prop:surgerycomp}.
Let $X(K)=Y\setminus \nu(K)$ (resp. $X(U)=Y\setminus \nu(U)$) be the bordered $3$-manifold obtained by removing a tubular neighborhood of $K$ (resp. $U$) from $Y$ and equipping the torus boundary with the parametrization corresponding to the framing. Fix bordered Heegaard diagrams $\HD_K$ and $\HD_U$ for $X(K)$ and $X(U)$ as in Figures \ref{fig:merge-H} and \ref{fig:split-H}, respectively. 
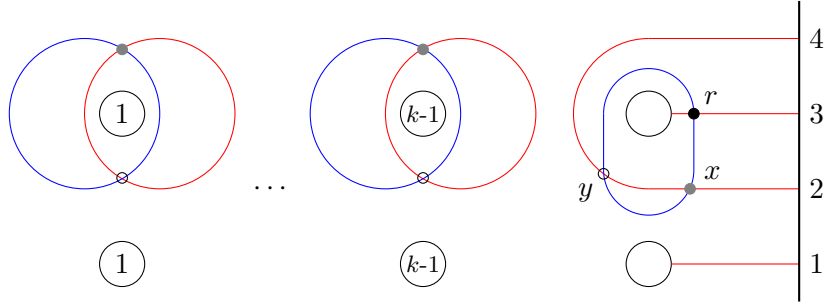
\begin{figure}
    \centering
    \begin{tikzpicture}
%
\def \r{.3}
\def \R{\r}
%
\foreach \y in {0,2} {
    \draw (0,\y) circle (\r cm);
    \draw[red] (\r,\y) --(2,\y);
    \draw[red] (0,\y+1) -- (2, \y+1);
}
    \draw[red] (0,1) arc (270:90:1 cm);
\draw[blue] (2*\R, 2) arc (0:180:2*\R cm);
\draw[blue] (2*\R, 1.25) arc (0:-180:2*\R cm);
\draw[blue] (2*\R, 1.25) -- (2*\R, 2);
\draw[blue] (-2*\R, 1.25) -- (-2*\R, 2);
\node[anchor=south west] at (2*\R,2) {$r$};
\draw[fill] (2*\R,2) circle (2 pt);
\node[anchor=south west] at (2*\R,1) {$x$};
\filldraw[gray] (2*\R-.05, 1) circle (2 pt);
\node[anchor=north east] at (-2*\R, 1.2) {$y$};
\draw (-2*\R, 1.2) circle (2 pt);
%
\begin{scope}[xshift=-3cm]
\foreach \y in {0,2} {
    \node at (0,\y) {\footnotesize $k$-$1$};
    \draw (0,\y) circle (\r cm);
}
\draw[red] (.5,2) circle (1cm);
\draw[blue] (-.5,2) circle (1cm);
\draw (0,2-2.85*\r) circle (2 pt);
\filldraw[gray] (0,2+2.85*\r) circle (2 pt);
\end{scope}
%
\node at (-5,1) (dots) {$\cdots$};
\begin{scope}[xshift=-7cm]
\foreach \y in {0,2} {
    \node at (0,\y) {$1$};
    \draw (0,\y) circle (\r cm);
}
\draw[red] (.5,2) circle (1cm);
\draw[blue] (-.5,2) circle (1cm);
\draw (0,2-2.85*\r) circle (2 pt);
\filldraw[gray] (0,2+2.85*\r) circle (2 pt);
\end{scope}
%
\draw[thick] (2,-.5) -- (2,3.5);
\foreach \y in {1,2,3,4}{
    \node[anchor=west] at (2,\y - 1) {\y};
}
\end{tikzpicture}
    \caption{The genus $k$ bordered Heegaard diagram $\Hext_K$ for the complement of $K$ in $Y$. Compare this with the diagram $\HD_\infty$ in Figure \ref{fig:CFD-tori}; we label the analogous intersection point $r$ as well.}
    \label{fig:merge-H}
\end{figure}
\begin{figure}
    \centering
    \begin{tikzpicture}
%
\def \r{.3}
\def \R{\r}
%
\foreach \y in {0,2} {
    \draw (0,\y) circle (\r cm);
    \draw[red] (\r,\y) --(2,\y);
    \draw[red] (0,\y+1) -- (2, \y+1);
}
    \draw[red] (0,1) arc (270:90:1 cm);
\draw[blue] (0,2-\r) arc (-180:0:\R cm);
\draw[blue] (2*\R,2-\r) -- (2*\R, 2);
\draw[blue] (2*\R, 2) arc (0:180:2*\R cm);
\draw[blue] (-2*\R,2) .. controls (-2*\R,1) and (0,1) .. (0,\r);
\node[anchor=south west] at (2*\R,2) {$b$};
\draw[fill] (2*\R,2) circle (2 pt);
\node[anchor=north east] at (-\R,1) {$a$};
\draw[fill] (-\R+.01,1+.03) circle (2 pt);
%
\begin{scope}[xshift=-3cm]
\foreach \y in {0,2} {
    \node at (0,\y) {\footnotesize $k$-$1$};
    \draw (0,\y) circle (\r cm);
}
\draw[red] (.5,2) circle (1cm);
\draw[blue] (-.5,2) circle (1cm);
\draw (0,2-2.85*\r) circle (2 pt);
\filldraw[gray] (0,2+2.85*\r) circle (2 pt);
\end{scope}
%
\node at (-5,1) (dots) {$\cdots$};
\begin{scope}[xshift=-7cm]
\foreach \y in {0,2} {
    \node at (0,\y) {$1$};
    \draw (0,\y) circle (\r cm);
}
\draw[red] (.5,2) circle (1cm);
\draw[blue] (-.5,2) circle (1cm);
\draw (0,2-2.85*\r) circle (2 pt);
\filldraw[gray] (0,2+2.85*\r) circle (2 pt);
\end{scope}
%
\draw[thick] (2,-.5) -- (2,3.5);
\foreach \y in {1,2,3,4}{
    \node[anchor=west] at (2,\y - 1) {\y};
}
\end{tikzpicture}
    \caption{The bordered Heegaard diagram $\HD_U$ for the complement of the unknot. }
    \label{fig:split-H}
\end{figure}
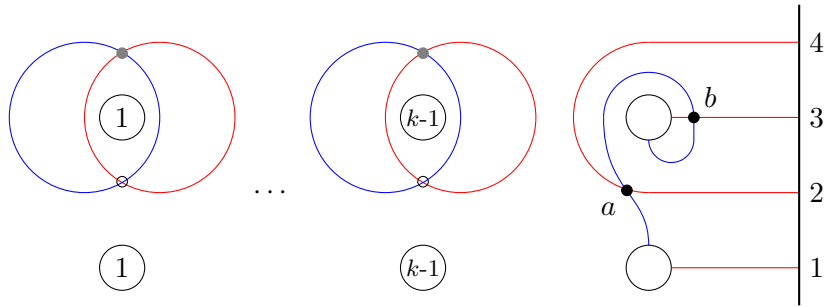

Thus, $\HD_K\cup \HD_{\infty}$ and $\HD_U\cup\HD_{\infty}$ are Heegaard diagrams for $Y$, while $\HD_{K}\cup\HD_{-1}$ and $\HD_U\cup\HD_{-1}$ are Heegaard diagrams for $Y_{-1}(K)$ and $Y_{-1}(U)$, respectively. Note that all of these diagrams are connected sums of genus $1$ Heegaard diagrams for $S^1\times S^2$ with at most one genus $1$ Heegaard diagram for $S^3$, where all the $S^1\times S^2$ diagrams have exactly two generators (in different Maslov gradings) and the base point is placed in the connected sum region. Thus, the involution map $\iota$ is equal to the identity map on each of  $\CFhat(\HD_K\cup\HD_{\infty})$, $\CFhat(\HD_K\cup\HD_{-1})$, $\CFhat(\HD_U\cup\HD_{\infty})$, and $\CFhat(\HD_U\cup\HD_{-1})$.  

Consequently, the surgery map from $\HFIhat(Y)$ to $\HFIhat(Y_{-1}(K))$ is induced by the following diagram:
\begin{center}
\begin{tikzcd}
\CFAhat(\HD_K) \DT \CFDhat(\HD_{\infty})
 \arrow{r}{\mathrm{Id} \DT  \phi} 
   \arrow{dr}{QG'}
  & \CFAhat(\HD_K) \DT \CFDhat(\HD_{-1})
       \\
Q \cdot \left(\CFAhat(\HD_K) \DT \CFDhat(\HD_{\infty})\right)
 \arrow{r}{\mathrm{Id} \DT  \phi} 
   &Q \cdot \left(\CFAhat(\HD_K) \DT \CFDhat(\HD_{-1})\right)
       \\
\end{tikzcd}
\end{center} 
The undrawn vertical maps $Q(\mathrm{id}+\iota)$  in the above square are zero maps, since $\iota=\mathrm{id}$. Since the map
$G'$ is a homotopy that makes the square commute, $G'$ is a chain map: \[d\circ G'+G'\circ d = \mathrm{Id} \DT  \phi \circ Q(\mathrm{id}+\iota) + Q(\mathrm{id}+\iota) \circ \mathrm{Id} \DT  \phi = 0. \]  Moreover, it follows from \cite[Proposition 6.1]{OzsvathSzaboss} that under the identifications of $\HFhat(Y)\cong \Lambda^*H_1(Y)$ and $\HFhat(Y_{-1}(K))\cong \Lambda^*H_1(Y_{-1}(K))$, the map $(\mathrm{Id} \DT \phi)_*$ is equal to the map induced by $q: H_1(Y) \to H_1(Y)/[K]$. So we need to check that $G'_*=0$. 

Similarly, for the unknot $U$ the induced map by  $\mathrm{Id} \DT \phi$ from $\HFhat(Y)\cong \Lambda^*H_1(Y)$ to $\HFhat(Y_{-1}(U))\cong \Lambda^*H_1(Y_{-1}(U))$ is equal to $i(\cdot)\wedge [K']$. Therefore, we need to check that $G'_*=i$.

We now verify the map $G'$ in the two cases $\HD = \HD_K$ and $\HD = \HD_U$. 
Recall that  $G'$ 
is defined by the composition 
\[G'=(\mathrm{Id} \DT  G)\circ (\h \DT \mathrm{Id}^2)\circ(\mathrm{Id} \DT \Omega \DT \mathrm{Id})\circ\eta\]
(Figure \ref{fig:involutive-surgery-diagram}).
We computed $\Omega$ in \S \ref{section-omega},  and $G$ is computed in \cite[Lemma 7.2]{HendricksLipshitz}. So to compute $G'$, we need to explicitly compute the homotopy equivalence
\[ \h:\CFAhat(\doublebar{\HD}) \DT \CFDAhat(\ovl{\AZ})\to\CFAhat(\HD)\]
for $\HD=\HD_K$ and $\HD_{U}$. 
By \cite[Lemma 4.2]{HendricksLipshitz}, it suffices to find a grading-preserving, non-nullhomotopic homomorphism. Below, we write the components of this $A_\infty$-module homomorphism as $\h = (\psi_1, \psi_2, \ldots)$. 

\begin{figure}
    \centering
    \begin{tikzpicture}
%
\def \r{.3}
\def \R{\r}
%

\begin{scope}[yscale=-1,yshift=-3cm]
\foreach \y in {0,2} {
    \draw (0,\y) circle (\r cm);
    \draw[blue] (\r,\y) --(2,\y);
    \draw[blue] (0,\y+1) -- (2, \y+1);
}
    \draw[blue] (0,1) arc (270:90:1 cm);
\draw[red] (2*\R, 2) arc (0:180:2*\R cm);
\draw[red] (2*\R, 1.25) arc (0:-180:2*\R cm);
\draw[red] (2*\R, 1.25) -- (2*\R, 2);
\draw[red] (-2*\R, 1.25) -- (-2*\R, 2);
%
\node[anchor=south west] at (2*\R,2) {$\doublebar r$};
\draw[fill] (2*\R,2) circle (2 pt);
\node[anchor=south west] at (2*\R,1) {$\doublebar x$};
\filldraw[gray] (2*\R-.05, 1) circle (2 pt);
\node[anchor=south east] at (-2*\R, 1.2) {$\doublebar y$};
\draw (-2*\R, 1.2) circle (2 pt);
\end{scope}
%
\draw[thick] (2,-.5) -- (2,3.5);
\foreach \y in {1,2,3,4}{
    \node[anchor=west] at (2,4-\y) {\y};
}
%
%
\begin{scope}[xshift=2.5cm]
    \draw[thick] (0,-.5) -- (0,3.5);
    \draw[thick] (0, 3.5) -- (4, 3.5) -- (4, -.5);
    \foreach \y in {1,2,3,4} {
        \draw[blue] (0,4-\y) -- (\y-.5, 3.5);
        \draw[red] (\y-.5, 3.5) -- (4, \y-1);
        \node[anchor=west] at (4, \y-1) {\y};
    }
    \foreach \x in {1,3}{
        \draw[fill] (\x-.5,3.5) circle (2pt);
        \node[anchor=south] at (\x-.5, 3.5) {$\io0^*$};
    }
    \foreach \x in {2,4}{
        \draw (\x-.5,3.5) circle (2pt);
        \node[anchor=south] at (\x-.5, 3.5) {$\io1^*$};
    }
\end{scope}
\end{tikzpicture}
    \caption{The diagrams for $\doublebar{\Hext_K}$ and $\AZbar$, with the extra copies of $S^1 \times S^2$ omitted.}
    \label{fig:merge-H-AZbar}
\end{figure}

First, suppose $\HD=\HD_K$, and see Figure \ref{fig:merge-H-AZbar} for the diagrams of $\ovl{\ovl{\HD_K}}$ and $\ovl{\AZ}$.   
In this case, the non-zero terms of $\hcomp_1$ are given by 
\begin{align*}
    \hcomp_1(\doublebar r \bt \iod1) 
    = \hcomp_1(\doublebar y \bt \rod{2,3}) 
    &=  y \\
    \hcomp_1(\doublebar r \bt \rod{2,4})
    &=  x \\
    \hcomp_1(\doublebar r \bt \rod{3,4})
    &=  r
\end{align*}
and the nonzero terms of $\hcomp_2$ are given by 

\begin{align*}
    \hcomp_2(\doublebar y \bt \rod{1,3} \otimes \ro{1,2})
    = \hcomp_2(\doublebar r \bt \rod{1,2} \otimes \ro{1,2})
    &=  x \\
    \hcomp_2(\doublebar y \bt \rod{1,3} \otimes \ro{1,3})
    = \hcomp_2(\doublebar r \bt \rod{1,2} \otimes \ro{1,3})
    &=  r \\
    \hcomp_2(\doublebar y \bt \rod{1,3} \otimes \ro{1,4})
    = \hcomp_2(\doublebar r \bt \rod{1,2} \otimes \ro{1,4})
    &=  y \\
\end{align*}
For all $i\ge 3$, we have $\hcomp_i=0$. It is easy to check that $\h$ is a homogeneous homotopy equivalence. 

\begin{figure}
    \centering
    \begin{tikzpicture}
%
\def \r{.3}
\def \R{\r}
%
\begin{scope}[yscale=-1,yshift=-3cm]
\foreach \y in {0,2} {
    \draw (0,\y) circle (\r cm);
    \draw[blue] (\r,\y) --(2,\y);
    \draw[blue] (0,\y+1) -- (2, \y+1);
}
    \draw[blue] (0,1) arc (270:90:1 cm);
\draw[red] (0,2-\r) arc (-180:0:\R cm);
\draw[red] (2*\R,2-\r) -- (2*\R, 2);
\draw[red] (2*\R, 2) arc (0:180:2*\R cm);
\draw[red] (-2*\R,2) .. controls (-2*\R,1) and (0,1) .. (0,\r);
\node[anchor=north west] at (2*\R,2) {$\doublebar b$};
\draw[fill] (2*\R,2) circle (2 pt);
\node[anchor=south east] at (-\R,1) {$\doublebar a$};
\draw[fill] (-\R+.01,1+.03) circle (2 pt);
\end{scope}
%
\draw[thick] (2,-.5) -- (2,3.5);
\foreach \y in {1,2,3,4}{
    \node[anchor=west] at (2,4-\y) {\y};
}
%
%
\begin{scope}[xshift=2.5cm]
    \draw[thick] (0,-.5) -- (0,3.5);
    \draw[thick] (0, 3.5) -- (4, 3.5) -- (4, -.5);
    \foreach \y in {1,2,3,4} {
        \draw[blue] (0,4-\y) -- (\y-.5, 3.5);
        \draw[red] (\y-.5, 3.5) -- (4, \y-1);
        \node[anchor=west] at (4, \y-1) {\y};
    }
    \foreach \x in {1,3}{
        \draw[fill] (\x-.5,3.5) circle (2pt);
        \node[anchor=south] at (\x-.5, 3.5) {$\io0^*$};
    }
    \foreach \x in {2,4}{
        \draw (\x-.5,3.5) circle (2pt);
        \node[anchor=south] at (\x-.5, 3.5) {$\io1^*$};
    }
\end{scope}
\end{tikzpicture}
    \caption{The diagrams for $\doublebar{\HD_U}$ and $\AZbar$, with the extra copies of $S^1 \times S^2$ omitted.}
    \label{fig:split-H-AZbar}
\end{figure}

Second, suppose $\HD=\HD_U$ and see Figure \ref{fig:split-H-AZbar} for the diagrams of $\ovl{\ovl{\HD_U}}$ and $\ovl{\AZ}$. In this case, the non-zero terms of $\h$ are given by 
\begin{align*}
    \hcomp_1(\doublebar b \bt \rod{3,4}) 
    &=  b \\
    \hcomp_1(\doublebar b \bt \iod1)
    &=  a
\end{align*}
and
\begin{align*}
    \hcomp_2(\doublebar a \bt \rod{1,3} \otimes \ro{1,3})
    = \hcomp_2(\doublebar a \bt \rod{2,3} \otimes \ro{2,3})
    &= b \\
    \hcomp_2(\doublebar a \bt \rod{1,3} \otimes \ro{1,4})
    = \hcomp_2(\doublebar a \bt \rod{2,3} \otimes \ro{2,4})
    &= a \\
    \hcomp_2(\doublebar a \bt \iod0 \otimes \ro{1,2})
    &= a.
\end{align*}

Finally, we are ready to compute $G'$, starting with the knot $K$.
\[\begin{split}
&\left(\mathrm{Id} \DT  G\right)\circ\left(\h \DT \mathrm{Id}^2\right)\circ\left(\mathrm{Id} \DT \Omega \DT \mathrm{Id}\right)\circ\eta(x|r)=\\
&\left(\mathrm{Id} \DT  G\right)\circ\left(\h \DT \mathrm{Id}^2\right)\circ\left(\mathrm{Id} \DT \Omega \DT \mathrm{Id}\right)(\ovl{\ovl{x}}|\io1|\ovl{\ovl{r}})=\\
&\left(\mathrm{Id} \DT  G\right)\circ\left(\h \DT \mathrm{Id}^2\right)\left(\ovl{\ovl{x}}|\rho^*_{2,3}|\rho_{3,4}|\ovl{\ovl{r}}+\ovl{\ovl{r}}|\rho_{3,4}^*|\iota_1|\ovl{\ovl{r}}+\ovl{\ovl{r}}|\rho_{2,4}^*|\rho_{3,4}|\ovl{\ovl{r}}\right)=\\
&\left(\mathrm{Id} \DT  G\right)(r|\iota_1|\ovl{\ovl{r}}+x|\rho_{3,4}|\ovl{\ovl{r}})=0.
\end{split}\]
 The diagrams contributing to the nonzero terms for $\left(\mathrm{Id} \DT \Omega \DT \mathrm{Id}\right)(\ovl{\ovl{x}}|\io1|\ovl{\ovl{r}})$ are depicted in Figure \ref{fig:id-Omega-id-eg1}.
 
\begin{figure}
    \centering
    \begin{tikzpicture}[xscale=.7]
\node (nw) at (-2,2) {$\doublebar x$};
\node (n) at (0,2) {$\io1$};
\node (ne) at (2,2) {$\doublebar r$};
\node (w) at (-2,-1) {$m_2$};
\node (o) at (0,0) {$\Omega_1^1$};
\node (e) at (2,1) {$\delta^0$};
\node (sw) at (-2,-2) {$\doublebar x$};
\node (s) at (0,-2) {$\iod0 \bt \io1$};
\node (se) at (2,-2) {$\doublebar r$};
%
\draw[dashed] (nw) -- (w) -- (sw);
\draw[dashed] (n) -- (o) -- (s);
\draw[dashed] (ne) -- (e) -- (se);
%
\draw[-stealth] (o) to 
node[above] {$\io1$} (w);
\end{tikzpicture}

\hspace{1cm}

\begin{tikzpicture}[xscale=.7]
\node (nw) at (-2,2) {$\doublebar x$};
\node (n) at (0,2) {$\io1$};
\node (ne) at (2,2) {$\doublebar r$};
\node (w) at (-2,-1) {$m_2$};
\node (o) at (0,0) {$\Omega_2^1$};
\node (e) at (2,1) {$\delta^1$};
\node (sw) at (-2,-2) {{\color{blue}$\doublebar x$}};
\node (s) at (0,-2) {{\color{blue}$\rod{2,3} \bt \ro{3,4}$}};
\node (se) at (2,-2) {$\doublebar r$};
\node (sw2) at (-2,-2.5) {{\color{red}$\doublebar r$}};
\node (s2) at (0,-2.5) {{\color{red}$\rod{3,4} \bt \io1$}};
%
\draw[dashed] (nw) -- (w) -- (sw);
\draw[dashed] (n) -- (o) -- (s);
\draw[dashed] (ne) -- (e) -- (se);
%
\draw[-stealth] (o) to 
    node[above left] {{\color{blue} $\io1$}}
    node[below right] {{\color{red} $\ro{2,3}$}}
    (w);
\draw[-stealth] (e) to
    node[above] {$\ro{2,4}$}
    (o);
\end{tikzpicture}

\hspace{1cm}
\begin{tikzpicture}[xscale=.7]
\node (nw) at (-2,3) {$\doublebar x$};
\node (n) at (0,3) {$\io1$};
\node (ne) at (2,3) {$\doublebar r$};
\node (w) at (-2,-1) {$m_2$};
\node (o) at (0,0) {$\Omega_2^1$};
\node (e1) at (2,1) {$\delta^1$};
\node (e2) at (2,2) {$\delta^1$};
\node (sw) at (-2,-2) {{$\doublebar r$}};
\node (s) at (0,-2) {{$\rod{2,4} \bt \ro{3,4}$}};
\node (se) at (2,-2) {$\doublebar r$};
%
\draw[dashed] (nw) -- (w) -- (sw);
\draw[dashed] (n) -- (o) -- (s);
\draw[dashed] (ne) -- (e2) -- (e1) -- (se);
%
\draw[-stealth] (o) to 
    node[above] {{$\ro{2,3}$}}
    (w);
\draw[-stealth] (e1) to
    node[below right] {$\ro{2,4}$}
    (o);
\draw[-stealth] (e2) to
    node[above left] {$\ro{1,4}$}
    (o);
\end{tikzpicture}
    \caption{The nonzero terms in the computation of $\left(\mathrm{Id} \DT \Omega \DT \mathrm{Id}\right)(\ovl{\ovl{x}}|\ovl{\ovl{r}})$.}
    \label{fig:id-Omega-id-eg1}
\end{figure} 
 
Moreover,
 \[\begin{split}
&\left(\mathrm{Id} \DT  G\right)\circ\left(\h \DT \mathrm{Id}^2\right)\circ\left(\mathrm{Id} \DT \Omega \DT \mathrm{Id}\right)\circ\eta(y|r)=\\
&\left(\mathrm{Id} \DT  G\right)\circ\left(\h \DT \mathrm{Id}^2\right)\circ\left(\mathrm{Id} \DT \Omega \DT \mathrm{Id}\right)(\ovl{\ovl{y}}|\io1|\ovl{\ovl{r}})=\\
&\left(\mathrm{Id} \DT  G\right)\circ\left(\h \DT \mathrm{Id}^2\right)\left(\ovl{\ovl{y}}|\rho_{2,3}^*|\rho_{3,4}|\ovl{\ovl{r}}\right)=\\
&\left(\mathrm{Id} \DT  G\right)(y|\rho_{3,4}|\ovl{\ovl{r}})=0.
\end{split}\]

Next, we compute $G'$ for the unknot $U$. 
\[\begin{split}
&\left(\mathrm{Id} \DT  G\right)\circ\left(\h \DT \mathrm{Id}^2\right)\circ\left(\mathrm{Id} \DT \Omega \DT \mathrm{Id}\right)\circ\eta(a|r)=\\
&\left(\mathrm{Id} \DT  G\right)\circ\left(\h \DT \mathrm{Id}^2\right)\circ\left(\mathrm{Id} \DT \Omega \DT \mathrm{Id}\right)(\ovl{\ovl{a}}|\io1|\ovl{\ovl{r}})=\\
&\left(\mathrm{Id} \DT  G\right)\circ\left(\h \DT \mathrm{Id}^2\right)\left(\ovl{\ovl{a}}|\rho_{2,3}^*|\rho_{3,4}|\ovl{\ovl{r}}+\ovl{\ovl{a}}|\rho_{3,4}^*|\iota_1|\ovl{\ovl{r}}+\ovl{\ovl{a}}|\iota_0^*|\iota_1|\ovl{\ovl{r}}\right)=\\
&\left(\mathrm{Id} \DT  G\right)(b|\rho_{2,4}|\ovl{\ovl{r}})=b|a.
\end{split}\]

Thus, $G'$ maps the top generator of $\HFhat(\HD_{U}\cup\HD_{\infty})$ to the top generator of $\HFhat(\HD_{U}\cup\HD_{-1})$. Therefore, $G'_*=i$.

\section{Involutive spectral sequence}

In this section, we use the bordered description of the branched double cover spectral sequence by Lipshitz, Ozsv\'ath, and Thurston \cite{LOTborderedss, LOTborderedss2} (see \S \ref{sec:borderedss}) to construct a spectral sequence converging to the involutive Heegaard Floer homology.  

\subsection{Construction of filtered complex}
\label{sec:involutivess}
Let $L\subset S^3$ be a link. As in \S \ref{sec:borderedss}, consider a diagram $D=B_0B_1\cdots B_{n+1}$ for $L$ such that $D$ is the plat closure of a braid on $2k$ strands with $n$ crossings. So, each $B_i$ is an elementary braid for $i=1,\cdots, n$, while $B_0$ and $B_{n+1}$ are $k$ cups and caps in a row, respectively. Each $\Sigma(B_i)$ is identified with a bordered $3$-manifold, and by Equation \ref{eq:pairing}, $\CFhat(\Sigma(L))$ is identified with a $\{0,1\}^n$-filtered chain complex. 

\begin{theorem}\label{thm:flr-iota}
Under the identification \ref{eq:pairing}, the involution $\iota$ on $\CFhat(\Sigma(L))$ is chain homotopic to a filtered map with respect to the $\{0,1\}^n$-grading.
\end{theorem}  

Using bordered Heegaard Floer homology, Hendricks and Lipshitz \cite{HendricksLipshitz} give a bordered extension of involutive Heegaard Floer homology.  Inspired by their work, we give a bordered description of the involution $\iota$ up to homotopy equivalence before proving Theorem \ref{thm:flr-iota}.

Recall that $\Sigma(B_0)$ and $\Sigma(B_{n+1})$ are bordered handlebodies of genus $k-1$, so that their boundaries are identified with $F(\ZZ)$ and $-F(\ZZ)$, respectively. Let $\HD_{B_0}$ and $\HD_{B_{n+1}}$ be bordered Heegaard diagrams for $\Sigma(B_0)$ and $\Sigma(B_{n+1})$ with $\bdy \HD_{B_0}=\ZZ$ and $\bdy \HD_{B_{n+1}}=-\ZZ$, respectively.

For $1\le i\le n$, $\Sigma(B_i)$ is the mapping cylinder of a Dehn twist along a simple closed curve $\gamma_i\subset F(\ZZ)$. Let $Y_{\Id}=[0,1]\times F(\ZZ)$. 
After removing a neighborhood $\nu(\{\frac{1}{2}\}\times\gamma_i)$ of the curve $\{\frac{1}{2}\}\times\gamma_i$ in $Y_\Id$, we parametrize the torus boundary  of $Y_{\Id}\setminus \nu(\{\frac{1}{2}\}\times\gamma_i)$  so that $\Sigma(B_i)$, $\Sigma(B_i^0)$ and $\Sigma(B_i^1)$ are obtained by filling this boundary with slopes $0$, $-1$ and $\infty$, respectively. Then, construct a triply bordered manifold from $Y_{\Id}\setminus \nu(\{\frac{1}{2}\}\times\gamma_i)$ by connecting the torus boundary to one of the boundary components identified with $F(\ZZ)$ with a framed arc. 
Let $\HD^i$ be a triply bordered Heegaard diagram for this triply bordered $3$-manifold.  By construction, $\HD_{B_i}=\HD^i\cup \HD_0$ is a bordered Heegaard diagram for $\Sigma(B_i)$. Recall $\HD_0$ denotes the standard bordered Heegaard diagram for the $0$-framed solid torus. 

In order to compute the involutive chain map $\iota$ on $\Sigma(L)$, we use the pairing theorem of bordered Floer homology to identify  $\CFhat(\Sigma(L))$ with a boxed tensor product of bordered Heegaard Floer modules, and then consider the composition of maps in Equation \ref{eq:binv}.

\begin{align} \label{eq:binv}
\CFhat(\Sigma(L))
\simeq\ &\CFAhat(\HD_{B_0}) \DT \CFDAAhat(\HD^1) \DT \CFDhat(\HD_0) \DT \cdots \DT \CFDhat(\HD_{B_{n+1}})\notag\\
\xrightarrow{\eta} \ & \CFAhat(\ovl{\ovl{\HD_{B_0}}}) \DT \CFDAAhat(\ovl{\ovl{\HD^1}}) \DT \CFDhat(\ovl{\ovl{\HD_0}}) \DT \cdots \DT \CFDhat(\ovl{\ovl{\HD_{B_{n+1}}}})\notag\\
=\ \ &\CFAhat(\ovl{\ovl{\HD_{B_0}}}) \DT [\mathrm{Id}] \DT \CFDAAhat(\ovl{\ovl{\HD^1}}) \DT [\mathrm{Id}] \DT \cdots \DT \CFDhat(\ovl{\ovl{\HD_0}})\notag\\
&\quad \quad  \DT [\mathrm{Id}] \DT \CFDhat(\ovl{\ovl{\HD_{B_{n+1}}}})\notag\\
\xrightarrow{\Omega_1} \ 
& \CFAhat(\ovl{\ovl{\HD_{B_0}}}) \DT \CFDAhat(\Id) \DT \CFDAAhat(\ovl{\ovl{\HD^1}}) \DT \CFDAhat(\Id) \DT \CFDhat(\ovl{\ovl{\HD_0}})\\
&\quad \quad  \quad  \DT \CFDAhat(\Id) \DT \cdots \DT \CFDhat(\ovl{\ovl{\HD_{B_{n+1}}}})\notag\\
\xrightarrow{\Omega_2} \ 
&\CFAhat(\ovl{\ovl{\HD_{B_0}}}) \DT \CFDAhat(\ovl{\AZ}) \DT \CFDAhat(\AZ) \DT \CFDAAhat(\ovl{\ovl{\HD^1}}) \DT \CFDAhat(\ovl{\AZ})\notag\\
&\quad \quad  \quad  \DT \CFDAhat(\AZ) \DT \CFDhat(\ovl{\ovl{\HD_0}}) \DT \CFDAhat(\ovl{\AZ}) \DT \cdots \DT \CFDhat(\ovl{\ovl{\HD_{B_{n+1}}}})\notag\\
\xrightarrow{\Psi_{0} \DT \Psi_1 \DT \widetilde{\Psi} \DT \cdots \DT \Psi_{n+1}}  \ &  \CFAhat(\HD_{B_0}) \DT \CFDAAhat(\HD^1) \DT \CFDhat(\HD_0) \DT \cdots \DT \CFDhat(\HD_{B_{n+1}})\notag\\
\simeq\ &\CFhat(\Sigma(L)).\notag
\end{align}

Here, $[\mathrm{Id}]=[\mathrm{Id}_{\mathcal{A}(\ZZ)}]$ is the type $DA$ identity bimodule of $\mathcal{A}(\ZZ)$ and $\Id=\Id_{\ZZ}$ is the standard bordered Heegaard diagram representing the identity cobordism from $F(\ZZ)$ to $F(\ZZ)$ (Section \ref{sec:identity}). The map $\Omega_1$ is defined by a homotopy equivalence between each copy of $[\mathrm{Id}]$ and $\CFDAhat(\Id)$. The map $\Omega_2$ is defined by the homotopy equivalence induced by a sequence of Heegaard moves from each copy of $\Id$ to $\ovl{\AZ}\cup\AZ$, as discussed in Section \ref{section-omega}. The chain homotopies $\widetilde{\Psi},\Psi_0$, and $\Psi_{n+1}$ are induced by a sequence of Heegaard moves from $\AZ\cup\ovl{\ovl{\HD_0}}$, 
 $\ovl{\ovl{\HD_{B_0}}}\cup\ovl{\AZ}$, and $\AZ\cup\ovl{\ovl{\HD_{B_{n+1}}}}$ to $\HD_0$, $\HD_{B_0}$, and $\HD_{B_{n+1}}$, respectively. Similarly, $\Psi_i$ is induced by a sequence of Heegaard moves from $\AZ\cup\ovl{\ovl{\HD^{i}}}\cup\ovl{\ovl{\AZ}}\cup\ovl{\ovl{\AZ}}$ to $\HD^i$ for any $1\le i\le n$. 

\begin{lemma}\label{lem:borinvolution} With the above notation, under the identification of $\CFhat(\Sigma(L))$ with \[\CFAhat(\HD_{B_0}) \DT \CFDAAhat(\HD^1) \DT \CFDhat(\HD_0) \DT \cdots \DT \CFDhat(\HD_{B_{n+1}})\] by the pairing theorems, the map $\Psi\circ\Omega\circ\eta$ defined in Equation \ref{eq:binv} is homotopic to the involutive chain map $\iota$ on $\CFhat(\Sigma(L))$ defined in \cite{HendricksManolescu}. Here, $\Omega=\Omega_2\circ\Omega_1$, and $\Psi=\Psi_0 \DT \Psi_1 \DT \widetilde{\Psi} \DT \cdots \DT \Psi_{n+1}$.
\end{lemma}

\begin{proof}
The proof is similar to the proof of \cite[Theorem 5.1]{HendricksLipshitz}.  
For simplicity, the following discussion will focus on the case of one elementary braid in the link diagram ($n=1$); the general case of $n$ elementary braid pieces will follow immediately.
We need to verify that up to homotopy, the following diagrams commute.

\begin{equation}\label{eq:eta-trimod-com}
    {\xymatrix{
    \CFAhat(\HD_{B_0}) \DT \CFDAAhat(\HD^1)  \DT  \CFDhat(\HD_0)  \DT \CFDhat(\HD_{B_2}) \ar[d]_\eta \ar[r] & \CFhat(.)\ar[d]_\eta\\
    \CFAhat(\ovl{\ovl{\HD_{B_0}}}) \DT \CFDAAhat(\ovl{\ovl{\HD^1}}) \DT  \CFDhat(\ovl{\ovl{\HD_0}}) \DT \CFDhat(\ovl{\ovl{\HD_{B_2}}})
    \ar[r] & \CFhat(.),
    }}
\end{equation}
where the horizontal arrows come from the pairing theorems for bordered Floer homology (for type $DA$ bimodules \cite[Theorems 11 and 12]{LOTbimodules}). In the right column of Equation~\ref{eq:eta-trimod-com}, the omitted terms are the obvious Heegaard diagrams from the pairing theorem, i.e.\ unions of the bordered Heegaard diagrams from the left column, for instance $\HD_{B_0}\cup\HD^1\cup\HD_0\cup\HD_{B_2}$ in the top-right term. Below, we will continue to use this notation.

\begin{equation}\label{eq:omega1-com}
    {\xymatrix{
    [\ovl{\ovl{\HD_{B_0}}}]  \DT  [\ovl{\ovl{\HD^1}}] \DT  [\ovl{\ovl{\HD_0}}]   \DT [\ovl{\ovl{\HD_{B_2}}}] \ar[d]_{\Omega_1} \ar[r] & \CFhat(. )\ar[d]\\
    [\ovl{\ovl{\HD_{B_0}}}]  \DT  [\Id] \DT  [\ovl{\ovl{\HD^1}}] \DT  [\Id] \DT  [\ovl{\ovl{\HD_0}}]  \DT  [\Id]  \DT [\ovl{\ovl{\HD_{B_2}}}]
    \ar[r] & \CFhat(.),
    }}
\end{equation}

\begin{equation}\label{eq:omega2-com}
    {\xymatrix{
    [\ovl{\ovl{\HD_{B_0}}}]  \DT  [\Id] \DT  [\ovl{\ovl{\HD^1}}] \DT  [\Id] \DT  [\ovl{\ovl{\HD_0}}]  \DT  [\Id]  \DT [\ovl{\ovl{\HD_{B_2}}}] \ar[d]_{\Omega_2} \ar[r] & \CFhat(. )\ar[d]\\
    {\begin{split}&[\ovl{\ovl{\HD_{B_0}}}]  \DT  [\AZbar]  \DT  [\AZ] \DT  [\ovl{\ovl{\HD^1}}] \DT  [\ovl{\AZ}]\\ & \DT  [\AZ] \DT  [\ovl{\ovl{\HD_0}}]  \DT  [\ovl{\AZ}]  \DT  [\AZ]  \DT [\ovl{\ovl{\HD_{B_2}}}]\end{split}}
    \ar[r] & \CFhat(.)
    }}
\end{equation} 

\begin{equation}\label{eq:psi-com}
    {\xymatrix{
    {\begin{split}&[\ovl{\ovl{\HD_{B_0}}}]  \DT  [\ovl{\AZ}]  \DT  [\AZ] \DT  [\ovl{\ovl{\HD^1}}] \DT  [\ovl{\AZ}]\\
    & \DT  [\AZ] \DT  [\ovl{\ovl{\HD_0}}]  \DT  [\ovl{\AZ}]  \DT  [\AZ]  \DT [\ovl{\ovl{\HD_{B_2}}}]\end{split}} \ar[d]_{ \Psi=\Psi_0 \DT \Psi_1 \DT \widetilde{\Psi} \DT \Psi_{2}} \ar[r] & \CFhat(. )\ar[d]\\
    \CFAhat(\HD_{B_0}) \DT \CFDAAhat(\HD^1)  \DT  \CFDhat(\HD_0)  \DT \CFDhat(\HD_{B_2})
    \ar[r] & \CFhat(.),
    }}
\end{equation} 
where the square brackets are shorthand for the appropriate bordered Floer module, i.e.\ $[\ovl{\ovl{\HD_{B_0}}}]$ denotes $\CFAhat(\ovl{\ovl{\HD_{B_0}}})$, while $[\Id]$ denotes $\CFDAhat(\Id)$, $[\ovl{\ovl{\HD^1}}]$ denotes $\CFDAAhat(\ovl{\ovl{\HD^1}})$, $[\ovl{\ovl{\HD_0}}]$ denotes $\CFDhat(\ovl{\ovl{\HD_0}})$, $[\ovl{\ovl{\HD_{B_2}}}]$ denotes $\CFDhat(\ovl{\ovl{\HD_{B_2}}})$, $[\AZ]$ denotes $\CFDAhat(\AZ)$, and $[\ovl{\AZ}]$ denotes $\CFDAhat(\ovl{\AZ})$. The horizontal arrows are chain homotopy equivalences from the pairing theorem, while the right vertical arrows are chain homotopy equivalences induced by a sequence of Heegaard moves. 
 
The proof of commutativity for Diagram \ref{eq:eta-trimod-com} is the similar to the proof for \cite[Diagram 5.2]{HendricksLipshitz}. Note that as proved by \cite{HendricksLipshitz}, the almost complex structures can be chosen such that this diagram commutes on the nose. 
  
The proof of commutativity for diagrams \ref{eq:omega2-com} and \ref{eq:psi-com} is similar to the proof for \cite[Diagram 5.4]{HendricksLipshitz} and \cite[Diagram 5.5]{HendricksLipshitz}, respectively, and it follows from \cite[Lemma 5.6]{HendricksLipshitz} and its bimodule version. 
  
Finally, the proof of commutativity for Diagram \ref{eq:omega1-com} is similar to the proof of commutativity for \cite[Diagram 5.3]{HendricksLipshitz}.
\end{proof}
  
\begin{proof}[Proof of Theorem \ref{thm:flr-iota}]
We recall from Section~\ref{sec:borderedss} some useful facts. After decomposing the link diagram of $L$ into elementary braids $B_i$, we view the Heegaard Floer chain complex of the branched double cover of the link $L$ as a boxed tensor product of bordered modules and bimodules via Equation \ref{eq:pairing}. The $\{0,1\}^n$-filtration on the right-hand side of Equation \ref{eq:pairing} is obtained by identifying each type $DA$ bimodule $\CFDAhat(\Sigma(B_i))$ with the mapping cone of a bimodule map 
\[f_i:\CFDAhat(\Sigma(B_i^1))\to\CFDAhat(\Sigma(B_i^0)),\]
where the superscript denotes the type of resolution of the crossing of $B_i$. 
This identification is obtained by chain homotopy equivalences  
\[\begin{split}\CFDAhat(\HD_{B_i})&\simeq\CFDAAhat(\HD^i) \DT \CFDhat(\HD_0)\simeq \CFDAAhat(\HD^i) \DT \cone(\phi)\\
&\simeq \cone\left(\mathrm{Id} \DT \phi:\CFDAhat(\HD^i\cup\HD_{\infty})\to \CFDAhat(\HD^i\cup\HD_{-1})\right)
\\
&\simeq \cone\left(f_i:\CFDAhat(\Sigma(B_i^1))\to\CFDAhat(\Sigma(B_i^0)) \right)
\end{split}\]
where the first homotopy equivalence follows from the bordered Floer pairing theorem, and the second follows from $\CFDhat(\HD_0)$ being homotopy equivalent to the mapping cone of a map $\phi:\CFDhat(\HD_{\infty})\to\CFDhat(\HD_{-1})$.

Further, one can analogously define a map $\ovl{\ovl{\phi}}$ between the appropriate modules such that $\CFDhat(\ovl{\ovl{\HD_0}})\simeq\cone(\ovl{\ovl{\phi}})$ and $\eta\phi=\ovl{\ovl{\phi}}\eta$.

By Lemma \ref{lem:borinvolution},  $\iota$ is homotopy equivalent to the chain map $\Psi\circ\Omega\circ\eta$ from Equation \ref{eq:binv}. Identifying every $\CFDhat(\HD_0)$ and $\CFDhat(\ovl{\ovl{\HD_0}})$ with $\cone(\phi)$ and $\cone({\ovl{\ovl{\phi}}})$, respectively, it is easy to see that $\eta$ and $\Omega$ preserve filtration.  By \cite[Lemma 7.2]{HendricksLipshitz}, the $i$-th copy of the chain homotopy $\widetilde{\Psi}$ from $\CFDAhat(\AZ) \DT \CFDhat(\ovl{\ovl{\HD_0}})$ to $\CFDhat(\HD_0)$ is homotopy equivalent to 
\begin{equation}
\label{eq:psi-filtration-preserving}
\begin{split}
\cone \Big(     \mathrm{Id}  \DT  \ovl{\ovl{\phi}}:\CFDAhat(\AZ) \DT \CFDhat(\ovl{\ovl{\HD_{\infty}}}
)
\to \CFDAhat(\AZ) \DT \CFDhat(\ovl{\ovl{\HD_{-1}}})\Big)\\
\xrightarrow{
\begin{bmatrix}\widetilde{\Psi}&0\\G&\widetilde{\Psi}\end{bmatrix}
    } 
    \cone \left( \phi : \CFDhat(\HD_{\infty}) \to \CFDhat(\HD_{-1}) \right)
\end{split}
\end{equation}
using the identification 
\begin{align*}
& \CFDAhat(\AZ)  \DT  \cone(\ovl{\ovl{\phi}}) \\
& \simeq  \ \cone\Big(\mathrm{Id} \DT \ovl{\ovl{\phi}}:\CFDAhat(\AZ) \DT \CFDhat(\ovl{\ovl{\HD_{\infty}}})
\to \CFDAhat(\AZ) \DT \CFDhat(\ovl{\ovl{\HD_{-1}}})\Big).
\end{align*}
Here, by an abuse of notation, $\widetilde{\Psi}$ denotes the chain homotopy induced by a sequence of Heegaard moves from $\AZ\cup\overline{\overline{\HD_{*}}}$ to $\HD_{*}$ for $*=-1$ or $\infty$. 

Since the top-right entry of the matrix in Equation \ref{eq:psi-filtration-preserving} is 0, $\Psi$ is filtration preserving, and therefore $\iota$ is also.
\end{proof}  

Thus, the involutive chain complex
\[
    \CFIhat(\Sigma(L)) = \cone\left(Q(1+\iota):\CFhat(\Sigma(L))\to Q\cdot\CFhat(\Sigma(L))[-1]\right)
\]
is $\{0,1\}^n$-filtered, induced by the $\{0,1\}^n$-filtration on $\CFhat(\Sigma(L))$ viewed as a boxed tensor product in Equation~\ref{eq:pairing}  as discussed in Section~\ref{sec:borderedss}. Here, we are assuming the link $L$ has a diagram $D$ with $n$ crossings. This filtration gives a spectral sequence with chain group  
\[\CFhat(\Sigma(D_v))\oplus Q\cdot \CFhat(\Sigma(D_v))[-1]\]
at each vertex $v\in\{0,1\}^n$ and converging to $\HFIhat(\Sigma(L))$.

\subsection{Identification of $(E_1, d_1)$ with the Bar-Natan complex}

In this section, we prove that the first page of the spectral sequence is isomorphic to the reduced Bar-Natan chain  complex for the mirror image of the link. 

\begin{proposition}\label{prop:vertex}
 Let $D$ be a diagram with $n$ crossings for the link $L$. Then, for any vertex $v\in \{0,1\}^n$ the $\BF[Q]/(Q^2)$-module at the vertex $v$ of the first page of the spectral sequence is isomorphic to $\HFIhat(\Sigma(D_v))$.
\end{proposition} 

\begin{proof} For each vertex $v\in \{0,1\}^n$ we have the chain group \[\CFhat(\Sigma(D_v))\oplus Q\cdot \CFhat(\Sigma(D_v))[-1].\] The complex $\CFhat(\Sigma(D_v))$ is identified with 

\begin{equation}\label{eq:vertex}
    \CFhat(\Sigma(D_v))\simeq\CFAhat(\HD_{B_0}) \DT \CFDAAhat(\HD^1) \DT \CFDhat(\HD_{w_1}) \DT \cdots \DT \CFDhat(\HD_{B_{n+1}})
\end{equation}
where $w_i=\infty$ if $v_i=1$, otherwise $w_i=-1$. In Equation \ref{eq:binv}, replace the $i$-th copy of $\HD_0$ and $\ovl{\ovl{\HD_0}}$ with $\HD_{w_i}$ and $\ovl{\ovl{\HD_{w_i}}}$, respectively, for every $1\le i\le n$ and define a chain map $\Psi_v\circ\Omega_v\circ\eta_v$ on $\CFhat(\Sigma(D_v))$. By an argument analogous to the proof of Lemma \ref{lem:borinvolution}, one can show that this map is chain homotopic to the Heegaard Floer involutive map on $\CFhat(\Sigma(D_v))$. On the other hand, it follows from the proof of Theorem \ref{thm:flr-iota} that $\Psi_v\circ\Omega_v\circ\eta_v$ is equal to the degree zero term of the involutive map on $\CFhat(\Sigma(L))$. Hence, on the $E_1$ page we get $\HFIhat(\Sigma(D_v))$ at each vertex $v\in\{0,1\}^n$.
\end{proof}

The next propositions show that the second page of the involutive  branched double cover spectral sequence is isomorphic to $\BNred(\ovl{L})$.
 
\begin{proposition}\label{prop:diff-surgery}
Let $v,v'\in\{0,1\}^n$ be a pair of vertices that differ at one index (that is, $v_i=1$ and $v'_i=0$ for some $i$, and $v_j = v_j'$ for $j \neq i$). Then, $\Sigma(D_{v'})$ is obtained by $-1$-surgery on a framed knot $\gamma_i\subset\Sigma(D_v)$. The differential of the first page of the spectral sequence is a map from $\HFIhat(\Sigma(D_v))$ to $\HFIhat(\Sigma(D_{v'}))$ which agrees with the surgery map defined in \cite[Theorem 7.1]{HendricksLipshitz}. More precisely, fixing the notation from \S \ref{sec:involutivess},  consider the bordered Heegaard diagram \[\HD_{v,v'}=\HD_{B_0}\cup\HD_{B_1}\cup\cdots\cup\HD^i\cup\HD_{B_{i+1}}\cup\cdots\cup\HD_{B_{n+1}}\]
for the complement of $\gamma_i$ in $\Sigma(D_v)$. Thus, $\HD_{v,v'}\cup\HD_{\infty}$ and $\HD_{v,v'}\cup\HD_{-1}$ are Heegaard diagrams for $\Sigma(D_v)$ and $\Sigma(D_{v'})$, respectively. Let $f=F \DT \mathrm{Id}$ denote the map from 
\[\CFAhat(\HD_{B_0}) \DT \CFDAAhat(\HD^1) \DT \cdots \DT \CFDAAhat(\HD^i) \DT \CFDhat(\HD_{\infty}) \DT \cdots\CFDhat(\HD_{B_{n+1}})\]
to $\CFAhat(\HD_{v,v'}) \DT \CFDhat(\HD_{\infty})$, where $F$ is the pairing map. Similarly, we define $f'$ by replacing $\infty$ with $-1$. Then, \[(f',f')\circ d_{v,v'}\simeq \begin{bmatrix}\mathrm{Id} \DT \phi&0\\ G'_{v,v'}&\mathrm{Id} \DT \phi\end{bmatrix}\circ (f,f)\]
where $d_{v,v'}$ is the part of differential on $\CFIhat(\Sigma(L))$ that maps $\CFIhat(\Sigma(D_v))$ to $\CFIhat(\Sigma(D_{v'}))$, and $\begin{bmatrix}\mathrm{Id} \DT \phi&0\\ G'_{v,v'}&\mathrm{Id} \DT \phi\end{bmatrix}$ is the involutive surgery map from $\CFIhat(\HD_{v,v'}\cup\HD_{\infty})$ to $\CFIhat(\HD_{v,v'}\cup\HD_{-1})$. Here, $G'_{v,v'}$ denotes the homotopy map $G'$ as defined in Section~\ref{sec:surgery}. 
\end{proposition} 

\begin{proof}
Note that $d_{v,v'}=\begin{bmatrix}d_{v,v'}^{11}&0\\\iota_{v,v'}&d_{v,v'}^{22}\end{bmatrix}$ where under the identification of Equation \ref{eq:vertex} for $\CFhat(\Sigma(D_v))$ and a similar identification for $\CFhat(\Sigma(D_{v'}))$
\[d_{v,v'}^{11}=d_{v,v'}^{22}=\mathrm{Id}^{\DT 2i} \DT \phi \DT \mathrm{Id}^{\DT 2n-2i+1}.\] 
Here, as before $\phi$ is the morphism from $\CFDhat(\HD_{w_i})=\CFDhat(\HD_{\infty})$ to $\CFDhat(\HD_{w'_i})=\CFDhat(\HD_{-1})$ whose mapping cone is $\CFDhat(\HD_0)$. Moreover, using the notation in the proof of Proposition \ref{prop:vertex}, we have $\iota_{v,v'}=\Psi_{v}^i\circ\Omega_v\circ\eta_v$ where $\Psi_v^i$ is obtained from $\Psi_v$ by replacing the $i$-th copy of $\widetilde{\Psi}$ with $G$.

By \cite[Lemma 2.3.3]{LOTbimodules} $(\mathrm{Id} \DT \phi)\circ f\simeq f'\circ d_{v,v'}^{11}$ and so we are left with checking that $G_{v,v'}\circ f\simeq f'\circ\iota_{v,v'}.$
Recall that, 
$G_{v,v'}=(\Psi_{v,v'} \DT  G)\circ\Omega_{v,v'}\circ\eta_{v,v'}$ is given by:
\begin{center}
\begin{align*}
\eta_{v,v'}:& \ \CFAhat(\HD_{v,v'}) \DT \CFDhat(\HD_{\infty})\longrightarrow \CFAhat(\overline{\overline{\HD_{v,v'}}}) \DT \CFDhat(\overline{\overline{\HD_{\infty}}})\\
\Omega_{v,v'}:& \  \CFAhat(\overline{\overline{\HD_{v,v'}}}) \DT \CFDhat(\overline{\overline{\HD_{\infty}}})\to \CFAhat(\overline{\overline{\HD_{v,v'}}}) \DT \CFDAhat(\mathbb{I}) \DT \CFDhat(\overline{\overline{\HD_{\infty}}})\\
&\to \CFAhat(\overline{\overline{\HD_{v,v'}}}) \DT \CFDAhat(\overline{\AZ}) \DT \CFDAhat(\AZ) \DT \CFDhat(\overline{\overline{\HD_{\infty}}})\\
\Psi_{v,v'} \DT  G:& \ \CFAhat(\overline{\overline{\HD_{v,v'}}}) \DT \CFDAhat(\overline{\AZ}) \DT \CFDAhat(\AZ) \DT \CFDhat(\overline{\overline{\HD_{\infty}}})\\
&\to \CFAhat(\HD_{v,v'}) \DT \CFDhat(\HD_{-1}).
\end{align*}
\end{center}
where $\Psi_{v,v'}$ from $\CFAhat(\overline{\overline{\HD_{v,v'}}}) \DT \CFDAhat(\overline{\AZ})$ to $\CFAhat(\HD_{v,v'})$ is induced by Heegaard moves from $\ovl{\ovl{\HD_{v,v'}}}\cup\overline{\AZ}$ to $\HD_{v,v'}$. Thus, we need to prove that 
\begin{equation}\label{eq:HLmaps}
(\Psi_{v,v'} \DT  G)\circ\Omega_{v,v'}\circ\eta_{v,v'}\circ f\simeq f'\circ\left(\Psi^i_v\circ\Omega_v\circ\eta_v\right)
\end{equation}

To simplify the equations, we assume that $n=i=1$ and use the bracket notation for the rest of the proof. The general case is similar. 

Let $\Omega_{v,1}$ and $\Omega_{v,1^c}$ be the maps 
\[\begin{split}
&\Omega_{v,1}: [\ovl{\ovl{\HD_{B_0}}}] \DT [\ovl{\ovl{\HD^1}}] \DT [\ovl{\ovl{\HD_{\infty}}}] \DT [\ovl{\ovl{\HD_{B_2}}}]\longrightarrow  [\ovl{\ovl{\HD_{B_0}}}] \DT [\ovl{\ovl{\HD^1}}] \DT [\ovl{\AZ}] \DT [\AZ] \DT [\ovl{\ovl{\HD_{\infty}}}] \DT [\ovl{\ovl{\HD_{B_{2}}}}]\\
&\Omega_{v,1^c}: [\ovl{\ovl{\HD_{B_0}}}] \DT [\ovl{\ovl{\HD^1}}] \DT [\ovl{\AZ}] \DT [\HD_{-1}] \DT [\ovl{\ovl{\HD_{B_2}}}]\longrightarrow [\ovl{\ovl{\HD_{B_0}}}] \DT [\ovl{\AZ}] \DT [\AZ] \DT [\ovl{\ovl{\HD^1}}] \DT [\ovl{\AZ}]\\
&\hspace{8.5cm} \DT [\HD_{-1}] \DT [\ovl{\AZ}] \DT [\AZ] \DT [\ovl{\ovl{\HD_{B_2}}}]\end{split}.\]
A similar argument as in the proof of commutativity for Equation \ref{eq:eta-trimod-com}, along with \cite[Lemma 2.3.3]{LOTbimodules} implies that
\[\ovl{\ovl{f'}} \circ (\mathrm{Id}^{3} \DT  G \DT \mathrm{Id})\circ\Omega_{v,1}\circ \eta_v \simeq (\mathrm{Id}^2 \DT  G)\circ\Omega_{v,v'}\circ\eta_{v,v'}\circ f,\] where $\ovl{\ovl{f'}}=\ovl{\ovl{F}} \DT \mathrm{Id}$ and $\ovl{\ovl{F}}$ is the bordered Floer homology pairing map. 

On the other hand, by \cite[Lemma 2.3.3]{LOTbimodules} the following diagram commutes up to homotopy.
\[
\xymatrix{
[\ovl{\ovl{\HD_{B_0}}}] \DT [\ovl{\ovl{\HD^1}}] \DT [\ovl{\ovl{\HD_{\infty}}}] \DT [\ovl{\ovl{\HD_{B_2}}}]\ar[d]_{\Omega_v}\ar[r]^{\Omega_{v,1}}& [\ovl{\ovl{\HD_{B_0}}}] \DT [\ovl{\ovl{\HD^1}}] \DT [\ovl{\AZ}] \DT [\AZ] \DT [\ovl{\ovl{\HD_{\infty}}}] \DT [\ovl{\ovl{\HD_{B_{2}}}}]\ar[d]_{\mathrm{Id}^3 \DT  G \DT \mathrm{Id}}\\
{\begin{split}&[\ovl{\ovl{\HD_{B_0}}}] \DT  [\ovl{\AZ}] \DT [\AZ] \DT [\ovl{\ovl{\HD^1}}] \DT [\ovl{\AZ}]\\ & \DT [\AZ] \DT [\ovl{\ovl{\HD_{\infty}}}] \DT [\ovl{\AZ}] \DT [\AZ] \DT [\ovl{\ovl{\HD_{B_{2}}}}]\end{split}}
\ar[d]_{\Psi^1_v}&[\ovl{\ovl{\HD_{B_0}}}] \DT [\ovl{\ovl{\HD^1}}] \DT [\ovl{\AZ}] \DT [\HD_{-1}] \DT [\ovl{\ovl{\HD_{B_2}}}]\ar[d]_{\Omega_{v,1^c}}\\
[\HD_{B_0}] \DT [\HD^1] \DT [\HD_{-1}] \DT [\HD_{B_2}]&\ar[l]_{\Psi_0 \DT \Psi_1 \DT \mathrm{Id} \DT \Psi_2} {\begin{split}&[\ovl{\ovl{\HD_{B_0}}}] \DT [\ovl{\AZ}] \DT [\AZ] \DT [\ovl{\ovl{\HD^1}}] \DT [\ovl{\AZ}]\\
& \DT [\HD_{-1}] \DT [\ovl{\AZ}] \DT [\AZ] \DT [\ovl{\ovl{\HD_{B_2}}}]\end{split}}
}
\]

So to prove Equation \ref{eq:HLmaps} we need to show that
 \[(\Psi_{v,v'} \DT \mathrm{Id})\circ\ovl{\ovl{f'}}\simeq f'\circ(\Psi_0 \DT \Psi_1 \DT \mathrm{Id} \DT \Psi_2)\circ\Omega^c_{v,1}.
 \]
 
 Next, consider the following diagram  
\[
\xymatrix{
[\ovl{\ovl{\HD_{B_0}}}] \DT [\ovl{\ovl{\HD^1}}] \DT [\ovl{\AZ}] \DT [\HD_{-1}] \DT [\ovl{\ovl{\HD_{B_{2}}}}]\ar[d]_{\Omega_{v,1^c}} \ar[r]& \CFhat(\ovl{\ovl{\HD_{v,v'}}}\cup\ovl{\AZ}\cup\HD_{-1})\ar[d]&\\
{\begin{split} &[\ovl{\ovl{\HD_{B_0}}}] \DT  [\ovl{\AZ}] \DT [\AZ] \DT [\ovl{\ovl{\HD^1}}] \DT [\ovl{\AZ}]\\ & \DT [\HD_{-1}] \DT [\ovl{\AZ}] \DT [\AZ] \DT [\ovl{\ovl{\HD_{B_{2}}}}]\end{split}} \ar[d]_{\Psi_0 \DT \Psi_1 \DT \mathrm{Id} \DT  \Psi_2} \ar[r] &\CFhat(\ovl{\ovl{\HD_{v,v'}'}}\cup \ovl{\AZ}\cup\HD_{-1})\ar[d]&\ar[ul] [\ovl{\ovl{\HD_{v,v'}}}] \DT [\ovl{\AZ}] \DT [\HD_{-1}]\ar[d]\\
[\HD_{B_0}] \DT [\HD^1] \DT [\HD_{-1}] \DT  [\HD_{B_0}] \ar[r]& \CFhat(\HD_{v,v'}\cup\HD_{-1})&\ar[l] [\HD_{v,v'}] \DT [\HD_{-1}]
}
\]
where $\HD_{v,v'}'=\HD_{B_0}\cup\ovl{\AZ}\cup\AZ\cup \HD^1\cup \ovl{\AZ}\cup\AZ\cup\HD_{B_2}$ and the vertical unlabeled arrows are induced by Heegaard moves or bordered Heegaard moves, while the rest of the unlabeled arrows are bordered Floer homology pairing maps. As in the proof of Lemma \ref{lem:borinvolution}, it follows from \cite[Lemma 5.6]{HendricksLipshitz} that the squares homotopy commute. Further, the same lemma \cite[Lemma 5.6]{HendricksLipshitz} along with the naturality of Heegaard Floer homology \cite{JuhaszThurston} implies that the pentagon homotopy commutes as well, and so we are done. 
\end{proof}

\begin{proposition}\label{prop:changebasis} For any vertex $v\in\{0,1\}^n$, there is an isomorphism $g_v$ of $\BF[Q]/(Q^2)$-modules between $\HFIhat(\Sigma(D_v))$ and $\CBNred(D_v)$ so that for every $v'\prec_1 v$ the following diagram commutes. 
\[\xymatrix{
\HFIhat(\Sigma(D_v))\ar[d]_{g_v}\ar[r]^{d_1}&\HFIhat(\Sigma(D_{v'}))\ar[d]_{g_{v'}}\\
\CBNred(D_v)\ar[r]&\CBNred(D_{v'})}
\]
Here, $d_1$ denotes the differential on the $E_1$ page of the spectral sequence from $\HFIhat(\Sigma(D_v))$ to $\HFIhat(\Sigma(D_{v'}))$, while the bottom horizontal arrow is the differential of the Bar-Natan complex $\CBNred(\ovl{D})$.
\end{proposition}

\begin{proof} 
Recall that at each vertex $v\in \{0,1\}^n$
\[\CBNred(D_v)=\widetilde{\CKh}(D_v)\otimes \BF[Q]/(Q^2)\]
which is the subspace of $\CKh(D_v)\otimes \BF[Q]/(Q^2)$ where the circle containing the based point is marked with $\vv_-$. 
In \cite[Proposition 6.2]{OzsvathSzaboss} $\widetilde{\CKh}(D_v)$ is identified with $\Lambda^*(H_1(\Sigma(D_v)))$ as follows.  

Suppose $D_v=\coprod_{i=0}^kS_i$ is a disjoint union of planar circles $S_i$ and $S_0$ contains the base point. Then, consider a pairwise disjoint collection of arcs $\gamma_1,\gamma_2,\cdots, \gamma_k$ such that $\gamma_i$ connects $S_0$ to $S_i$. Then, the preimage of each $\gamma_i$ is a circle $\widetilde{\gamma_i}$ in $\Sigma(D_v)\cong \#^k(S^1\times S^2)$ and the homology classes $[\widetilde{\gamma_1}], [\widetilde{\gamma_2}],\cdots,[\widetilde{\gamma_k}]$ give a basis for $H_1(\Sigma(D_v))$. For a primitive class $\xi\in\Lambda^*\langle [\widetilde{\gamma_1}], [\widetilde{\gamma_2}],\cdots,[\widetilde{\gamma_k}]\rangle$ let $\phi(\xi)\in \widetilde{\CKh}(D_v)$  be the generator that assigns $\vv_+$ to the circle $S_0$ containing the base point, and assigns $\vv_-$ to the circle $S_i$ if $[\widetilde{\gamma}_i]$ appears in $\xi$ and $\vv_+$ to the circle $S_i$ otherwise. This map descends to an isomorphism from $\Lambda^*\langle [\widetilde{\gamma_1}], [\widetilde{\gamma_2}],\cdots,[\widetilde{\gamma_k}]\rangle$ to $\widetilde{\CKh}(D_v)$.  Thus, $\phi\otimes\mathrm{Id}$ gives an isomorphism between $\Lambda^*\langle [\widetilde{\gamma_1}], [\widetilde{\gamma_2}],\cdots,[\widetilde{\gamma_k}]\rangle\otimes \BF[Q]/(Q^2)$ and $\CBNred(D_v)$. 

Suppose $v,v'\in\{0,1\}^n$ and $v'\prec_1 v$. Further, assume $D_{v'}$ is obtained from $D_v$ by splitting $S_k$. Propositions \ref{prop:vertexmod}, \ref{prop:surgerycomp}, and \ref{prop:diff-surgery} imply that the diagram 
\begin{center}
\begin{tikzcd}
  \HFIhat(\Sigma(D_v)) \arrow[r, "d_{1}"] \arrow[d, "\cong"]
    & \HFIhat(\Sigma(D_{v'})) \arrow[d, "\cong"] \\
  \Lambda^*\langle [\widetilde{\gamma_1}], [\widetilde{\gamma_2}],\cdots,[\widetilde{\gamma_k}]\rangle \otimes \BF[Q]/(Q^2) \arrow[r, "\delta"]
&  \Lambda^*\langle [\widetilde{\gamma_1}], [\widetilde{\gamma_2}],\cdots,[\widetilde{\gamma_k}], [\widetilde{\gamma_k}']\rangle\otimes \BF[Q]/(Q^2) \end{tikzcd}
\end{center}
commutes, where $\delta(\xi)=\xi\wedge [\widetilde{\gamma}_k']+Q\xi$ for any $\xi\in \Lambda^*\langle [\widetilde{\gamma_1}], [\widetilde{\gamma_2}],\cdots,[\widetilde{\gamma_k}]\rangle$. 
Here, $\widetilde{\gamma}_k'$ is the preimage of an arc $\gamma_k'$ connecting $S_k$ to $S_{k+1}$ in $D_{v'}$ in $\Sigma(D_{v'})$, as in Figure \ref{fig:basis}. 

\begin{figure}
    \centering
        \begin{tikzpicture}
%
\draw (0,0) circle (.5cm);
\draw (0,2) circle (.5cm);
%
\draw[thick, red] (0,.5) -- node[right]{$\gamma_k$} (0,1.5) ;
%
\node[draw=none]  at (0,-.5) (pt) {$*$};
%
\draw[-stealth] (1,1) -- node[above]{split} (3,1);
%
\begin{scope}[xshift=4cm]
\draw (0,0) circle (.5cm);
\draw (0,2) circle (.5cm);
\draw (2,2) circle (.5cm);
%
\draw[thick, red] (0,.5) -- node[right]{$\gamma_k$} (0,1.5) ;
\draw[thick, red] (0.5,2) -- node[above]{$\gamma_{k}'$} (1.5,2);
\draw[thick, red] (1.65, 1.65) -- node[right]{$\gamma_{k+1}$} (0.35,0.35);
%
\node[draw=none]  at (0,-.5) (pt) {$*$};
\end{scope}
\end{tikzpicture}
    \caption{}
    \label{fig:basis}
\end{figure}

Define a change of basis map (of $\BF[Q]/(Q^2)$-modules) from $\Lambda^*\langle [\widetilde{\gamma_1}],\cdots,[\widetilde{\gamma_k}], [\widetilde{\gamma_k}']\rangle\otimes \BF[Q]/(Q^2)$ to $\Lambda^*\langle [\widetilde{\gamma_1}],\cdots,[\widetilde{\gamma_k}], [\widetilde{\gamma}_{k+1}]\rangle\otimes \BF[Q]/(Q^2)$ by sending:
\begin{align}\label{Eq:chbasis}
    \xi &\mapsto \xi \nonumber\\
      [\widetilde{\gamma_k}] \wedge \xi  &\mapsto   [\widetilde{\gamma_k}] \wedge \xi  \nonumber\\
    [\widetilde{\gamma_k}']\wedge\xi &\mapsto  [\widetilde{\gamma_k}]\wedge\xi+ [\widetilde{\gamma}_{k+1}]\wedge\xi \\
    [\widetilde{\gamma_k}]\wedge[\widetilde{\gamma_k}']\wedge\xi &\mapsto  [\widetilde{\gamma_k}]\wedge[\widetilde{\gamma}_{k+1}]\wedge\xi+Q [\widetilde{\gamma_k}]\wedge\xi\nonumber
\end{align}
where $\xi\in\Lambda^*\langle [\widetilde{\gamma_1}],\cdots,[\widetilde{\gamma_{k-1}}]\rangle$
(compare this with \cite[Lemma 4]{Lin}). 
It is easy to check that composing $\delta$ with this change of basis map it coincides with the Bar-Natan split map $\Delta$.

Conversely, suppose $D_{v'}$ is obtained from $D_v$ by merging two components. For instance, assume $D_v=\amalg_{i=0}^{k+1}S_i$ and $D_{v'}$ is obtained by merging $S_k$ and $S_{k+1}$. As before, Propositions
\ref{prop:vertexmod}, \ref{prop:surgerycomp}, and \ref{prop:diff-surgery} imply that the diagram 
\begin{center}
\begin{tikzcd}
  \HFIhat(\Sigma(D_v)) \arrow[r, "d_1"] \arrow[d, "\cong"]
    & \HFIhat(\Sigma(D_{v'})) \arrow[d, "\cong"] \\
  \Lambda^*\langle [\widetilde{\gamma_1}], [\widetilde{\gamma_2}],\cdots,[\widetilde{\gamma_k}],[\widetilde{\gamma_k}']\rangle \otimes \BF[Q]/(Q^2) \arrow[r, "m'"]
&  \Lambda^*\langle [\widetilde{\gamma_1}], [\widetilde{\gamma_2}],\cdots,[\widetilde{\gamma_k}]\rangle\otimes \BF[Q]/(Q^2)
\end{tikzcd}
\end{center}
commutes, where $m'(\xi)=\xi$ and $m'([\widetilde{\gamma}_k']\wedge\xi)=0$ for every $\xi\in \Lambda^*\langle [\widetilde{\gamma_1}],\cdots,[\widetilde{\gamma_k}]\rangle$. 
It is easy to check that composing $m'$ with the change of basis in Equation \ref{Eq:chbasis} is the merge map $m$ for Bar-Natan homology. 
\end{proof}

\section{Examples}

We provide examples where the involutive branched double cover spectral sequence collapses immediately, and one where it does not.

\subsection{Khovanov $\delta$-thin knots}

The $\delta$-grading on Khovanov homology is given by $\delta = q - 2h$
where $h$ is the homological grading and $q$ is the quantum grading. 
A knot $K \subset S^3$ is \emph{(Khovanov) $\delta$-thin} if its reduced Khovanov homology is supported on a single $\delta$-grading. This class of knots includes all quasi-alternating knots \cite{Manolescu-Ozsvath-QA}, for example.

If a knot $K$ is $\delta$-thin, then its total dimension is determined by its Jones polynomial, and we have 
\[
    \dim_{\BF} \Khred(K) = | \det(K) |
\]
where $\det$ is the knot determinant. 
On the other hand, a 3-manifold $Y$ is an \emph{$L$-space} if $\HFhat(Y)$ is as small as possible, i.e.\ 
\[
    \dim_{\BF} \HFhat(Y) 
    = |H_1(Y)|,
\]
the order of the first homology group of $Y$.
When $Y = \Sigma(K)$ is the branched double cover of a knot, we have $|H_1(Y)| = |\det(K)|$. 
By duality properties of Khovanov homology, 
$\dim \Khred(K) = \dim \Khred(\overline K)$, 
and $\dim_{\BF} \BNred(\overline K) = \dim_{\BF} \BNred(K)$; furthermore, $\det(K) = \det(\overline K)$. 
We then have 
\[
    |\det(K)| \leq \dim_{\BF} \HFhat(\Sigma(K)) \leq \dim_{\BF} \Khred(\overline K) = |\det(\overline K)| = |\det(K)|
\]
(\cite[Corollary 1.2]{OzsvathSzaboss}), and so the  Ozsv\'ath-Szab\'o's link surgeries spectral sequence collapses immediately.

Hendricks and Manolescu show that if $Y$ is an $L$-space, then rank of $\HFI^+(Y, \fraks)$ (over $\BF[U,Q]/(Q^2)$) is the same as the rank of $\HFplus(Y, \fraks)$ (over $\BF[U]$) \cite[Corollary 4.8]{HendricksManolescu}.
When $U = 0$, we have
\[ 
    \HFIhat(Y,\fraks) \cong \HFhat(Y,\fraks) \otimes \BF[Q]/(Q^2). 
\]
In particular, if $K$ is Khovanov $\delta$-thin, then 
\[
    \dim_{\BF} \HFIhat(\Sigma(K)) 
    = 2\,\dim_{\BF} \HFhat(\Sigma(K)) 
    = 2\,|\det(K) |
    = \dim_{\BF}(\Khred(K)\otimes_{\BF} \BF[Q]/(Q^2)).
\]

We are now ready to prove the following.
\begin{theorem}
If $K$ is a $\delta$-thin knot, then the involutive link surgeries spectral sequence 
collapses immediately.
\begin{proof}
For any link $L$, the Bar-Natan--Lee spectral sequence
\[ 
    \Khred(\overline L)\otimes \BF[Q]/(Q^2) 
    \abuts
    \BNred(\overline L)
\]
first computes homology with respect to the Khovanov differentials on $\CKhred(D) \otimes \BF[Q]/(Q^2)$, and then computes homology with respect to the additional differentials given by the Bar-Natan perturbation (see \S \ref{sec:BN-complex-intro}).
Therefore 
\[ 
    \dim_{\BF} \Khred(\overline L)\otimes \BF[Q]/(Q^2) \geq 
    \dim_{\BF} \BNred(\overline L). 
\]

Thus for a $\delta$-thin knot $K$, 
\[
    \dim_{\BF} \BNred(\overline K) 
    \leq  \dim_{\BF} \Khred(\overline K)\otimes \BF[Q]/(Q^2) 
    = \dim_{\BF} \HFIhat(\Sigma(K))
    \leq \dim_{\BF} \BNred(\overline K),
\]
so $\dim_{\BF} \BNred(\overline K)$ is in fact equal to $\dim_{\BF} \HFIhat(\Sigma( K))$; i.e., the spectral sequence must collapse immediately.
\end{proof}
\end{theorem}

\subsection{A nontrivial spectral sequence}

In \cite[\S 6.8]{HendricksManolescu}, Hendricks and Manolescu compute $\HFI^+$ of the Brieskorn sphere
\[ 
     -\Sigma(2,3,7) 
    = S^3_{+1}(-T_{2,3})
    =\Sigma(-T_{3,7}),
\]
where $S^3_{+1}(-T_{2,3})$ is the $+1$-surgery of $S^3$ along the left-handed trefoil, and $\Sigma(-T_{3,7})$ is the branched double cover of $S^3$ along the mirror of the torus knot $T_{3,7}$.  

The complex $\CFIhat(\Sigma(-T_{3,7}))$ is the kernel of the map $U: \CFI^+(\Sigma(-T_{3,7})) \to \CFI^+(\Sigma(-T_{3,7}))$, and is shown in Figure~\ref{fig:brieskorn}. 
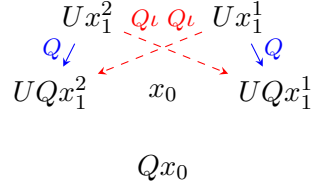
\begin{figure}[ht]
    \centering
    \begin{tikzpicture}[scale=0.5]
\node (x12) at (-2,2) {$Ux_1^2$};
\node (x11) at (2,2) {$Ux_1^1$};
\node (Qx12) at (-3,0) {$UQx_1^2$};
\node (Qx11) at (3,0) {$UQx_1^1$};
\node (x0) at (0,0) {$x_0$};
\node (Qx0) at (0,-2) {$Qx_0$};
\draw[red, dashed, -stealth] 
    (x12) to node[above, pos=0.2]{{\footnotesize $Q\iota$}} (Qx11);
\draw[red, dashed, -stealth] 
    (x11) to node[above, pos=0.2]{{\footnotesize $Q\iota$}} (Qx12);
\draw[blue, -stealth] 
    (x12) to node[left, pos=0.25]{{\footnotesize $Q$}} (Qx12);
\draw[blue, -stealth] 
    (x11) to node[right, pos=0.25]{{\footnotesize $Q$}} (Qx11);
\end{tikzpicture}
    \caption{The involutive complex $\CFIhat(\Sigma(-T_{3,7}))$ for $(+1)$-surgery on the left-handed trefoil (cf.\ \cite[Figure 12]{HendricksManolescu}). The solid blue arrows are $Q\cdot id$ maps, whereas the red dashed arrows are $Q \cdot \iota$ maps. There are no other differentials.}    \label{fig:brieskorn}
\end{figure}
The reader may check $\dim_{\BF}\HFIhat(\Sigma(-T_{3,7})) =4$.

To compute $\BNred(T_{3,7})$, one can use Lewark's \texttt{khoca} Khovanov homology calculator \cite{lewark-khoca}, by computing equivariant Khovanov ($\essl_2$-) homology over $\BF_2$ with potential $X^2 + bX + a$ where $a = 0$. Here, $b=-Q \ (=Q)$. The program outputs the Khovanov polynomial over $\BF = \BF_2$ along with the higher differentials in the Bar-Natan--Lee spectral sequence. 
After setting $Q^2=0$, we obtain Figure \ref{fig:T37}, and compute that $\dim_{\BF}\BNred(T_{3,7}) =14.$
Since $\dim_{\BF} E_2$ is strictly greater than $\dim_{\BF} E_\infty$, the spectral sequence does not collapse immediately.
\begin{figure}
    \centering
    \begin{tikzpicture}[xscale=1.2, yscale=.8]
\def\r{.5};
\def\s{.4};
\draw (-1-\r, 32*\r-\r)--(0-\r, 30*\r-\r);
\node at (-1,15-\s) {$q$};
\node at (-1+\s,15) {$h$};
\foreach \x in {0,...,9}{
    \draw (\x-\r, 5-\r) -- (\x-\r, 14+\r);
    \node at (\x, 15) {\x};
}
    \draw (10-\r, 5-\r) -- (10-\r, 14+\r);
\foreach \y in {10,12,...,28}{
    \draw (0-\r, .5*\y-\r) -- (9+\r, .5*\y-\r);
    \node at (-1, .5*\y) {\y};
}
    \draw (0-\r, 15-\r) -- (9+\r, 15-\r);
\node at (9,28*\r) {$y_9$};
\node at (8,24*\r) {$y_8$};
\node at (7,24*\r) {$y_7$};
\node (y6) at (6,22*\r) {$y_6$};
\node at (5,22*\r) {$y_5$};
\node at (4,18*\r) {$y_4$};
\node at (3,18*\r) {$y_3$};
\node (y2) at (2,16*\r) {$y_2$};
\node at (0,12*\r) {$y_0$};
\node at (9,28*\r-1) {$Qy_9$};
\node at (8,24*\r-1) {$Qy_8$};
\node (Qy7) at (7,24*\r-1) {$Qy_7$};
\node at (6,22*\r-1) {$Qy_6$};
\node at (5,22*\r-1) {$Qy_5$};
\node at (4,18*\r-1) {$Qy_4$};
\node (Qy3) at (3,18*\r-1) {$Qy_3$};
\node at (2,16*\r-1) {$Qy_2$};
\node at (0,12*\r-1) {$Qy_0$};
\draw[->, red, thick] (y2)--(Qy3);
\draw[->, red, thick] (y6)--(Qy7);
\end{tikzpicture}
    \caption{The visualized output of \texttt{khoca} on $T_{3,7}$, after setting $Q^2=0$. The $y_i$ (for $i$ in the range of supported homological $h$ gradings) generate $\Khred(T_{3,7})\otimes \BF[Q]/(Q^2)$ as a module over $\BF[Q]/(Q^2)$. The red arrows are higher differentials in the Bar-Natan--Lee spectral sequence of the form $y_i \mapsto Q\cdot y_{i+1}$; there are no longer differentials since $Q^2=0$.}
    \label{fig:T37}
\end{figure}

\bibliographystyle{alpha}
\bibliography{main}

\end{document}


\begin{tikzpicture}[every text node part/.style={align=center}]
\def\dx{6}
\def\dy{2}
\node (01) at (0*\dx,1*\dy) {$
       \CFAhat(\HD_{X})
       \DT 
        \CFDhat(\HD_\infty)
         $};
\node (11) at (1*\dx,1*\dy) {$
       \CFAhat(\HD_{X})
       \DT 
        \CFDhat(\HD_{-1})
    $};

\node (00) at (0*\dx,0*\dy) {$
       Q\left(\CFAhat(\HD_{X})
       \DT 
        \CFDhat(\HD_\infty)\right)
    $};
\node (10) at (1*\dx,0*\dy) {$
      Q\left( \CFAhat(\HD_{X})
       \DT 
        \CFDhat(\HD_{-1})\right)
    $};

\draw[->] (01) -- (11) node[midway,above] {{\footnotesize $\idmap^2 \DT \phi$}};
\draw[->] (00) -- (10) node[midway,above] {{\footnotesize $\idmap^2 \DT \phi$}};
%
%
%
%
%
\draw[->] (01) -- (00) node[midway,left] {{\footnotesize $Q(id+\iota)$}};
\draw[->] (11) -- (10) node[midway,left] {{\footnotesize $Q(id+\iota)$}};
%
\draw[->, dashed] (01) -- (10) node[pos=.4,above right] {{\footnotesize $QG'$}};

\end{tikzpicture}